\documentclass[12pt]{amsart}%
\usepackage{amsfonts}
\usepackage{amsmath}
\usepackage{amssymb}
\usepackage{graphicx}
\usepackage{color}
\makeatletter
\@addtoreset{equation}{section}
\makeatother

\newtheorem{theorem}{Theorem}[section]

\newtheorem{corollary}[theorem]{Corollary}

\newtheorem{lemma}[theorem]{Lemma}

\newtheorem{proposition}[theorem]{Proposition}
\newtheorem{remark}[theorem]{Remark}

\begin{document}

\title[A sharp trace Adams' inequality]{A sharp  trace Adams' inequality in $\mathbb{R}^{4}$ and Existence of the extremals}
\author{Lu Chen, Guozhen Lu and Maochun Zhu}
\address{Key laboratory of Algebraic Lie Theory and Analysis of Ministry of Education,
School of Mathematics and Statistics, Beijing Institute of Technology, Beijing 100081, P. R. China}
\email{chenlu5818804@163.com}
\address{Department of Mathematics\\
University of Connecticut\\
Storrs, CT 06269, USA}
\email{guozhen.lu@uconn.edu}
\address{School of mathematical sciences\\
Jiangsu University\\
Zhenjiang, 212013, P. R. China\\}
\email{zhumaochun2006@126.com}

\thanks{The first author was partly supported by the National Key Research and Development Program (No.
2022YFA1006900) and National Natural Science Foundation of China (No. 12271027). The second author was supported partly by a collaboration grant and Simons Fellowship from the Simons Foundation. The third author was supported by Natural Science Foundation of China (12071185). }

\maketitle

\begin{abstract}
Let $\Omega\subseteq
\mathbb{R}^{4}$ be a bounded domain  with smooth boundary $\partial\Omega$. In this paper, we establish the following sharp form of the trace Adams' inequality in $W^{2,2}(\Omega)$ with
 zero mean value and zero Neumann boundary condition:
\begin{equation*}
S({\alpha})=\underset{\int_{\Omega}udx=0,\frac{\partial u}{\partial\nu}|_{\partial\Omega}=0,\Vert\Delta u\Vert_{2}\leq{1}}{\underset
{u\in{W^{2,2}(\Omega)\setminus\{0\}}}{\sup}}\int_{\partial \Omega} e^{\alpha u^{2}}d\sigma<\infty
\end{equation*}
holds if and only if $ \alpha\leq12\pi^2$.

Moreover, we prove a classification theorem for the
solutions of a class of nonlinear boundary value problem of bi-harmonic equations on the half space $\mathbb{R}^4_{+}$. With this classification result, we can show that $S({12\pi^2})$ is attained by using the blow-up analysis and capacitary estimate. As an application, we prove a sharp trace Adams-Onofri type inequality  in general four dimensional bounded domains with smooth boundary.

\end{abstract}

\maketitle {\small {\bf Keywords:} Adams Trace inequality; Blow-up analysis; Extremal function; Capacity estimate.
 \\

{\bf 2010 MSC.} 35J60
35B33 46E30 }

\section{Introduction}

Let $\Omega\subseteq%
\mathbb{R}
^{n}$ be a bounded domain  with smooth boundary $\partial\Omega$ and $W^{m,p}\left(  \Omega\right) $ denote the usual Sobolev space:
the completion of $C^{\infty}\left(  \bar{\Omega}\right)  $ under the norm %
$$\parallel\cdot\parallel=\big(\sum_{|\alpha|\leq m}\int_{\Omega}|D^{\alpha}u|^pdx\big)^{\frac{1}{p}}.$$ If $1<p<n/m$, the classical Sobolev embedding asserts that $W^{m,p}\left(
\Omega\right)  \hookrightarrow L^{p^{\ast}}\left(  \Omega\right)  $ for
$p^{\ast}=\frac{np}{n-mp}$, and $W^{m,p}\left(
\Omega\right)  \hookrightarrow L^{p^{\star}}\left(  \partial\Omega\right)  $ for
$p^{\star}=\frac{(n-1)p}{n-mp}$. However, when $p=n/m$, it is known that both
$W^{m,p}\left( \Omega\right)  \hookrightarrow L^{\infty}\left(
\Omega\right)  $ and $W^{m,p}\left(  \Omega\right)  \hookrightarrow L^{\infty}\left(
\partial\Omega\right)  $  fail.
\medskip

It it known that the analogue of optimal Sobolev embedding for $W_{0}^{m,n/m}\left(  \Omega\right)$ (the Sobolev space consisting of functions  vanishing on the boundary $\partial\Omega$) is given by
the famous Trudinger-Moser inequality $\left(  m=1\right)$ (\cite{Mo},
\cite{Tru}) and Adams' inequality $\left(  m>1\right)$ (\cite{A}), which can be stated in the following form:

\begin{equation}
\underset{\Vert\Delta^{\frac{m}{2}}u\Vert_{\frac{n}{m}}\leq{1}}{\underset
{u\in{W_{0}^{m,\frac{n}{m}}(\Omega)\setminus\{0\}}}{\sup}}\int_{\Omega}\exp(\alpha
|u(x)|^{\frac{n}{n-m}})dx\left\{
\begin{array}
[c]{l}%
\leq c\left\vert \Omega\right\vert, \text{ if }\alpha\leq{\alpha(n,m),}\\
=+\infty,\text{ \ if }\alpha>{\alpha(n,m),}%
\end{array}
\right.  \label{adams}%
\end{equation}
where
\[
\alpha(n,m)=%
\begin{cases}
\frac{n}{\omega_{n-1}}[\frac{\pi^{n/2}2^{m}\Gamma(\frac{m+1}{2})}{\Gamma
(\frac{n-m+1}{2})}]^{\frac{n}{n-m}},\text{\text{when}}~m~\text{\text{is}
\text{odd,}}\\
\frac{n}{\omega_{n-1}}[\frac{\pi^{n/2}2^{m}\Gamma(\frac{m}{2})}{\Gamma
(\frac{n-m}{2})}]^{\frac{n}{n-m}},\text{\text{when}}~m~\text{\text{is}
\text{even,}}%
\end{cases}
\]
here $\omega_{n-1}$ denotes the $n-1$ dimensional surface measure of
the unit ball in $%
\mathbb{R}
^{n}$.
So far, the Trudinger-Moser-Adams inequalities (\ref{adams}) have been generalized
in many other directions such as the Trudinger-Moser inequalities on the unbounded domains,
compact Riemannian manifolds, C-R spheres, hyperbolic spaces, Heisenberg
groups, Hardy-Adams type inequalities on hyperbolic spaces, etc, to just name a few we refer the interested readers to \cite{Adachi-Tanaka, CL2,Chenluzhu2, ChenW,Li 1,cohn-lu, do, Fontana, Lam,LaLu4,LamLuTang-NA, Liluzhu, li-lu-zhu, liruf,LLWY, LuTang-ANS, LuYangQ1, LuYangQ2, ManciniSandeepTintarev, RS, Yang1}, and the
references therein.
\medskip

An interesting  problem related to the Trudinger-Moser-Adams inequalities
 lies in investigating the existence of extremal functions. Carleson and Chang \cite{c-c} first established the existence of extremals for Trudinger-Moser inequalities on the unit ball through symmetrization rearrangement inequality combining with the ODE technique.
After that, the existence of extremals was proved for any bounded
domains in $\mathbb{R}^n$ (see \cite{Flucher}, \cite{lin}, \cite{Adimurthi}). One can also see \cite{Li2,Li 1, liruf, ZhuJ} for existence of extremals for the Trudinger-Moser inequalities on compact Riemannian manifold, unbounded domains, and see \cite{lu-yang 1,DelaTorre,Chenluzhu} for the existence of extremals for Adams' inequalities in bounded and unbounded domains. We note that the Trudinger-Moser-Adams inequalities on the Sobolev spaces $W^{m,\frac{n}{m}}\left(  \Omega\right)$ without the Dirichlet boundary condition have also been established, the interested readers can refer to the work \cite{c-c1,Leckband,Cianchi1, lu-yang 2, Tarsi}, and the references therein.
\vskip 0.1cm

In this paper, we are interested in the borderline case of Sobolev trace inequality in $W^{m,\frac{n}{m}}(\Omega)$. As mentioned above, from the Sobolev embedding we know that $W^{m,\frac{n}{m}}\left(  \Omega\right)  \hookrightarrow L^{q}\left(
\partial\Omega\right)$ for any $q\in[1,\infty)$, but not for $q=\infty$. (see e.g., Maz'ya's book). Adams \cite{AD1} showed that  $W^{m,\frac{n}{m}}\left( \Omega  \right)$ can be imbedded into the Orlicz space ${L_\phi }(\partial\Omega )$, with $\phi(t)  = {\rm{exp}}\left( {{{\left| t \right|}^{\frac{n}{{n - m}}}} - 1} \right)$(see also \cite{Mazya}). The first optimal trace inequality of Moser type on $\partial\Omega$ was obtained in (\cite{Chang-Marshall}) in two dimensional disk $D$ for functions with zero boundary mean value. Namely,
\begin{equation*}\label{trace}
\underset{\int_{D}\left|\nabla u\right|^{2} dx=1, \int_{\partial D} u d \sigma=0}{\underset
{u \in W^{1,2}(D)\backslash\{0\}}{\sup}}\int_{\partial D} e^{\alpha u^{2}} d\sigma<+\infty, \text{iff } \alpha\leq\pi.
\end{equation*}
Using the technique of blow-up analysis, the authors of \cite{Liliu} extended the result of \cite{Chang-Marshall} to general bounded domains  $\Omega\subset \mathbb{R}^2$ and obtained the existence of corresponding extremals. The author of \cite{Yang} obtained another sharp form of trace Trudinger-Moser  inequality for functions with mean value zero. For more general trace Trudinger-Moser inequality in high dimensional case, one can see the work of  \cite{Cianchi}. For trace Trudinger-Moser and trace Adams inequalities on unbounded domains such as the half spaces, we refer to the work  \cite{CLYZ}.
\vskip0.1cm

The main purpose of this paper is to study the second order trace Adams inequality and the existence of extremals. Set $$\mathcal{H}=\{u\in W^{2,2}(\Omega):\|\Delta u\|_2\leq 1,\int_{\Omega}udx=0,\frac{\partial u}{\partial \nu}\mid_{\partial\Omega}=0\}.$$ Our main results read as follows:
\begin{theorem}\label{mainre}
Let $\Omega\subseteq \mathbb{R}^4$ be a bounded smooth domain with smooth boundary $\partial\Omega$. Then if $\alpha\leq12\pi^2$, we have
\begin{equation}\label{tracethm}
S(\alpha):=\underset{u\in\mathcal{H}\setminus\{0\}}{{\sup}}\int_{\partial \Omega} e^{\alpha u^{2}}d\sigma<\infty.
\end{equation}
The constant $12\pi^2$ is sharp in the sense that if $\alpha >12\pi^2$, then the supremum $S(\alpha)$ is infinity. Moreover, the supremum is attained if $\alpha\leq12\pi^2$.
\end{theorem}
As an immediate consequence of Theorem \ref{mainre}, we have the following result when $\Omega$ is a four dimensional ball $\mathbb{B}^4$ and $\mathbb{S}^3$ is its boundary.

\begin{corollary}\label{trace-ball}
If $\alpha\leq12\pi^2$, we have
\begin{equation}\label{tracethm2}
S(\alpha, \mathbb{B}^4):=\underset{u\in\mathcal{H}\setminus\{0\}}{{\sup}}\int_{\mathbb{S}^3} e^{\alpha u^{2}}d\sigma<\infty.
\end{equation}
The constant $12\pi^2$ is sharp in the sense that if $\alpha >12\pi^2$, then the supremum $S(\alpha, \mathbb{B}^4)$ is infinity. Moreover, the supremum $S(\alpha, \mathbb{B}^4)$ is attained if $\alpha\leq12\pi^2$.

\end{corollary}

As an application of Theorem \ref{mainre}, we can derive the following trace Adams-Onofri type inequality on any four dimensional bounded domains.

\begin{theorem}\label{Adamsonifri}Assume that  $\Omega\subseteq \mathbb{R}^4$ is a bounded smooth domain with smooth boundary $\partial\Omega$.
For any $u\in W^{2,2}(\Omega)$ with $\frac{\partial u}{\partial\nu}=0$ on $\partial\Omega$, there exists a constant $C$ such that
$$\frac{1}{48\pi^2}\int_{\Omega}|\Delta u|^2dx+\frac{1}{|\Omega|}\int_{\Omega}udx-\log\big(\int_{\partial\Omega}e^{u}d\sigma\big)\geq C.$$
\end{theorem}

\begin{remark} The first order Sobolev trace inequality  was due to Escobar \cite{Es} and Beckner \cite{Be}. The second order and higher order Sobolev trace inequalities were established by Ache and Chang \cite{Ache} and Q. Yang \cite{YangQ} respectively.
 The Moser-Onofri and Adams-Onofri inequalities on the sphere can be obtained by using the endpoint differentiation argument (see Beckner's work in \cite{Be}). However, this method cannot be used to establish the sharp trace Adams-Onofri inequality on general four dimensional bounded domains due to its absence of conformal invariance. Our Theorem \ref{Adamsonifri} offers a weak version of such Moser-Onofri type inequality on general bounded domains in $\mathbb{R}^4$ from the sharp Adams trace inequality obtained in Theorem \ref{mainre}.
\end{remark}

The general strategy we use here is exploiting the  blow-up analysis.
We first prove the subcritical trace Adams inequalities and the existence of extremals by using the sharp subcritical Adams inequalities involved with zero mean value.  Then, we take a sequence $\alpha
_{k}\rightarrow12\pi^{2}$ and find a  maximizing sequence
$\{u_{k}\}_k \subset W^{2,2}\left(
\Omega\right)  $ for $S\left(  12\pi^2\right)  $. If $u_{k}$ is bounded in
$L^{\infty}\left(
\Omega\right)  $, i.e. $c_{k}:=\underset{x\in%
\Omega}{\max}\left\vert u_{k}(x)\right\vert <\infty$, we can easily show that
$u_{k}$ converges to a function $u\,$in\ $W^{2,2}\left(
\Omega\right)  $ by elliptic estimates. If $c_{k}\rightarrow
+\infty$, i.e. the blow-up arises, we apply the blow-up analysis method to
analyze the asymptotic behavior of $u_{k}$ near and far away from the blow-up point $p\in\partial\Omega$, and derive an upper bound for the trace Adams
functional:
\begin{equation}\label{addupper}
S\left(  12\pi^2\right)  \leq  \left| \partial\Omega  \right| + 2\pi^2 e^{12\pi^2 A_p-\frac{3}{4}},
\end{equation}
where $A_p$ is the value at $p$ of the trace of the regular part of the Green
function for the operator $\Delta^{2}+\frac{1}{|\Omega|}$. Finally, we construct a function
sequence in $\mathcal{H}$ to show that the upper bound can actually be surpassed, which implies that the concentration phenomenon will not happen.
\vskip 0.1cm

Neither the blow-up strategy in the study of second order Adams inequality with Dirichlet boundary condition (see \cite{lu-yang 1}) and nor the blow-up method for the first order trace Trudinger-Moser inequality and existence of their extremals (see \cite{Liliu,   Yang}) can be easily generalized to second order trace Adams inequality case.
In the following, we will introduce the main difference between the proof of the second order trace Adams inequality and those  for the second order Adams inequality with Dirichlet boundary condition and the trace Trudinger-Moser blow-up analysis (see [31],[46]). We will also outline the elements of novelty when we carry out the blow-up procedure for our second order trace Adams inequality.
\vskip0.1cm

First of all, the main difference between the blow-up analysis for the Adams inequality   and that for the Adams trace inequality blowing-up is the location of the blow-up points. For the former, the blow-up points must be at some interior points, while for the latter the blow-up points must lie on the boundary of  $\Omega$, which leads to the situation where  the related Euler-Lagrange equations of the maximizer sequence are some bi-harmonic equations with the Neumann boundary condition and the analysis of asymptotic behavior near and far away from the blow-up points is totally different.

Second, unlike the first order case, one cannot show that the maximum point $x_k$ of $u_k$ lies on the boundary $\partial\Omega$ due to the lacking of maximum principle of the bi-harmonic operator $\Delta^2$. We stress that without this fact one cannot analyze the asymptotic behavior of $u_k$ near the blow up points $p\in\partial\Omega$ if the concentration phenomenon occurs. To overcome this difficulty, we will make use of the assumption that $\frac{\partial u_k}{\partial \nu}=0$ on $\partial\Omega$ to show that there exists some point $\tilde{x}_k\in\partial\Omega$ such that
$$|u_k(\tilde{x}_k)-u_k(x_k)|=o_k(1).$$ This important observation allows us to choose the maximum point $x_k$ on the boundary $\partial\Omega$.
\vskip0.1cm

Third, when we try to analyze the asymptotic behavior of $u_{k}$ near the blow up point $p$, a crucial
step is to classify the solutions to the following Liouville equation on the half space $\mathbb{R}^4_+:=\{x=(x',t)|x'\in \mathbb{R}^3,t>0\}$: \begin{equation*}
\begin{cases}
\Delta^2\psi=0,\  & x\in \mathbb{R}_{+}^{4},\\
\frac{\partial\Delta\psi}{\partial t}=\exp(24\pi^2 \psi),\ & x\in \partial\mathbb{R}_{+}^{4}, \\
\psi(0)=\sup \psi=0,\\
 \frac{\partial\psi}{\partial t}=0, \ \ & x\in \partial\mathbb{R}_{+}^{4}.
\end{cases}
\end{equation*}

In order to achieve this goal, we need to prove an important local estimate for $\Delta u_{k}$:
\begin{equation}\label{addlo}
\int_{B_{\rho}(x_k)\bigcap\Omega}|u_k\Delta u_{k}|dx\leq C\varrho^{2},%
\end{equation}
when $\rho$ is small. For this, we first rewrite $u_k\Delta u_{k}$
in terms of the Riesz potentia, then by using the Hardy-Littlewood-Sobolev inequality on the compact manifold with the boundary and the boundedness in $L\log^{\frac{1}{2}}L(\Omega)$ of $\Delta^2u_k$, we can show that $u_k\Delta u_{k}$ is bounded in the Lorentz space $L^{2,\infty}(\Omega)$, which implies \eqref{addlo}. Applying this local estimate and a careful computation, we can show that the solution $\psi$ must take the form as
$$\psi=-\frac{1}{8\pi^2}\log\left( \left(1+\left({\frac{\pi}{2}}\right)^{\frac{2}{3}}t   \right)^2+  \left({\frac{\pi}{2}}\right)^{\frac{4}{3}}|x'|^2   \right)+\frac{1}{2^{\frac{8}{3}}\pi^{\frac{4}{3}}}\frac{t}{  \left(1+\left({\frac{\pi}{2}}\right)^{\frac{2}{3}}t  \right)^2+  \left({\frac{\pi}{2}}\right)^{\frac{4}{3}}|x'|^2   }.$$
\vskip0.1cm

Fourth, when we try to obtain the upper bound for the trace Adams inequality if the concentration phenomenon occur, an important step consists in finding the sharp lower bounds of the integral of $|\Delta u_k|^2$ on some annular regions when we carry out the capacity estimates. In the earlier work of \cite{Liliu,Yang}, this can be achieved by comparing the energy of $u_k$ with the quantity \begin{equation}
\min _{u \in \{u| u(R_1)=a,u(R_2)=b\}} \int_{\left\{R_{1} \leq|x| \leq R_{2}\right\}\bigcap \mathbb{R}^2_+}\left|\nabla u\right|^{2} dx,
\end{equation}
whose extremal function is some harmonic function which can be explicitly obtained by solving some equation on the half space. However, in the second order case, finding the explicit expression of the corresponding extremal appears to be very hard.  In this work, we will compute the upper bound by directly comparing the Dirichlet energy of $u_k$ with some bi-harmonic function in the half annular region. In our situation the boundary of the upper half annular region involves some part of $\partial\Omega$ where $u_k$ is not vanishing. This will add a lot of trouble in the comparison of the corresponding calculations since the asymptotic behavior of $u_k$ cannot be obtained in this half annular region. In order to avoid the complicated computations on $\partial\Omega$, we will modify the bi-harmonic function to cancel the integral on the boundary $\partial \Omega$ (see Subsection \ref{Neck}).
\vskip0.1cm

Finally, since $\mathcal{H}$ requires that the test  functions not only satisfy $\|\Delta u\|^2\leq 1$, but also satisfy $\frac{\partial u}{\partial \nu}=0$ on $\partial\Omega$, this makes the construction of test functions more complicated when we try to show that the concentration upper bound can be surpassed.
\medskip

This paper is organized as follows. Section 2 is devoted to proving the sharp subcritical trace Adams inequality, and show the existence of extremals;  in Section 3, we show that the maximizing
sequence must concentrate around the blow-up point when the blow up arises. Moreover, we analyze the asymptotic behavior of the maximizing sequence  near and far away from the blow-up point, and derive an upper bound for the trace Adams functional; in Section 4, we prove the existence of extremals by constructing a proper test function sequence.

\section{The best constant for the trace Adams inequality}

In this section, we prove that the best constant in Theorem \ref{mainre} is $12\pi^2$. First, we recall the following subcritical Adams' inequalities for functions with mean value zero. Throughout this section, we let $\Omega\subseteq \mathbb{R}^4$ be a bounded smooth domain with smooth boundary $\partial\Omega$. We also recall that
 $$\mathcal{H}=\{u\in W^{2,2}(\Omega):\|\Delta u\|_2\leq 1,\int_{\Omega}udx=0,\frac{\partial u}{\partial \nu}\mid_{\partial\Omega}=0\}.$$

\begin{lemma}[\cite{Hang1},Theorem 3.2]\label{subcritical}
For any $\varepsilon>0$, we have
 \begin{equation}\label{sub-Tru}
\sup_{u\in\mathcal{H}\backslash\{0\}}\int_{\Omega}\exp((16\pi^2-\varepsilon)|u|^{2})dx<\infty.
\end{equation}

\end{lemma}

Next, we further prove the following

\begin{lemma}\label{best constant}
Set $\alpha_{2}=\sup \left\{\alpha: \sup _{u \in \mathcal{H}} \int_{\Omega} e^{\alpha u^{2}}<+\infty\right\} $. Then $\alpha_{2}=16\pi^2$.
\end{lemma}
\begin{proof}
By Lemma \ref{subcritical}, we know $\alpha_{2}\geq16 \pi^2$. In order to prove the lemma, we only need to show $\alpha\leq 16\pi^2$.
Taking any $p\in\partial\Omega$, for any $0<\rho<\delta$, we use the notation $B_{\rho}=B_{\rho}(p)$ and set
$$
u_{k}(x):= \begin{cases}\sqrt{\frac{1}{16 \pi^{2}} \log \frac{1}{R_{k}}}-\frac{|x|^{2}}{\rho^2\sqrt{4 \pi^{2} R_{k} \log \frac{1}{R_{k}}}}+\frac{1}{\sqrt{4 \pi^{2} \log \frac{1}{R_{k}}}}, & \text { if } x\in B_{\rho\sqrt[4]{R_{k}}}\bigcap\Omega, \\ \frac{1}{\sqrt{ \pi^{2} \log \frac{1}{R_{k}}}} \log \frac{\rho}{|x|}, & \text { if } x\in \left(B_{\rho}\setminus B_{\rho\sqrt[4]{R_{k}}}\right)\bigcap\Omega, \\ \eta_{k}(|x|), & \text { if } x\in \left(B_{\delta}\setminus B_{\rho}\right)\bigcap\Omega,\end{cases}
$$
where $\left\{R_{k}\right\}_{k \geq 1} \subset \mathbb{R}^{+}, R_{k} \searrow 0$, and $\eta_{k}$ satisfies $\left.\frac{\partial \eta_{k}}{\partial \nu}\right|_{\partial B_{\rho}}=-\frac{1}{\rho\sqrt{ \pi^{2} \log \frac{1}{R_{k}}}}$,$\left.\frac{\partial \eta_{k}}{\partial \nu}\right|_{\partial B_{\delta}}=0$, $\left.\eta_{k}\right|_{\partial B_{\rho}}=\left.\eta_{k}\right|_{\partial B_{\delta}}=0$,
and $\eta_{k}$, $\Delta \eta_{k}$ are all $O\left(1 / \sqrt{\log 1 / R_{k}}\right)$.
Since $u_k$ is radial, we can choose some function $\varphi_k(x)$ such that $$\frac{\partial \varphi_k}{\partial \nu}=\frac{\partial u_k}{\partial \nu}=o_\delta(1)\ \text{for } x\in \partial\Omega, $$ $\varphi_k(x)=o_\delta(1) $  and $\left\|\Delta \varphi_{k}\right\|_{2}^{2}=o_\delta(1)$ as $\delta\rightarrow0.$
\vskip0.1cm

For some fixed $r>\delta$, set $$
U_{k}(x):= \begin{cases} u_k-\varphi_k, & \text { if } x\in B_{\delta}\bigcap\Omega, \\ t_k\phi_k, & \text { if } x\in \left(B_{r}\setminus B_{\delta}\right)\bigcap\Omega, \end{cases}
$$
where $\phi_k $ is a smooth function such that $supp (\phi_k) \subset B_{r}\setminus B_{\delta}$, $\frac{\partial \phi_k}{\partial \nu}\big|_{\partial\Omega}=0$, and $t_k$ is selected such that $\int_{\Omega}U_{k}(x)dx=0$. Easy computation directly gives
$$
\left\|U_{k}\right\|_{2}^{2}=O_{k,r}(1), \quad \left\|\Delta U_{k}\right\|_{2}^{2}=1+O_{k,r}(1).
$$
Normalizing ${U}_k$ by $\tilde{U}_k=\frac{U_k}{\left\|\Delta U_{k}\right\|}$, we have  $\tilde{U}_k\in \mathcal{H}$. Then it follows that  for any fixed $\alpha>16\pi^2$, there exists some $\varepsilon_0>0,$ such that $$\int_{\Omega} e^{\alpha \tilde{U}_k^{2}}\geq \int_{\Omega\bigcap B_{\rho\sqrt[4]{R_{k}} }} e^{\alpha \tilde{U}_k^{2}}\geq c\rho^4 e^{\varepsilon_0 \log \frac{1}{R_{k}} }\rightarrow \infty,$$
as $k\rightarrow \infty$, and the proof is finished.
\end{proof}

Based on Lemma \ref{best constant}, we can show that the best constant of the inequality (\ref{tracethm}) is $12\pi^2$.
\begin{lemma}
Set $I_{\alpha}(u)=\int_{\partial \Omega} e^{\alpha u^{2}}d\sigma$. Then we have
$$
\sup _{u \in\mathcal{ H}\setminus\{0\}} I_{\alpha}(u)<+\infty \text { for } \alpha<12\pi^2 \text { and }  \sup _{u \in \mathcal{H}\setminus\{0\}} I_{\alpha}(u)=+\infty \text { for } \alpha>12\pi^2.
$$
\end{lemma}
\begin{proof}
Taking a smooth vector field $\vec{\nu}(x)$ whose restriction on $\partial\Omega$ is the outward unit normal vector field. Using the divergence theorem and Sobolev embedding theorem, we derive that for any $\varepsilon>0$, there holds
\begin{align}
  \int_{\partial \Omega} e^{(12\pi^2-\varepsilon) u^{2}}d\sigma &=\int_{\Omega} \operatorname{div}\left(\vec{\nu}(x) e^{(12\pi^2-\varepsilon) u^{2}}\right) dx \nonumber\\&=\int_{\Omega}\left(\operatorname{div}(\vec{\nu}(x))+2(12\pi^2-\varepsilon) u\langle\vec{\nu}(x), \nabla u\rangle\right) e^{(12\pi^2-\varepsilon) u^{2}}dx \nonumber \\
& \leq c\left(1+\int_{\Omega}|\nabla u \| u| e^{(12\pi^2-\varepsilon) u^{2}}\right)dx  \nonumber\\
& \leq c\left(1+\|\nabla u\|_{L^{4}(\Omega)}\|u\|_{L^{p}(\Omega)}\left\|e^{(12\pi^2-\varepsilon) u^{2}}\right\|_{L^{\frac{16 \pi^2-\varepsilon} {12\pi^2-\varepsilon}}(\Omega)}\right) \nonumber\\
& \leq c+c\left(\|\Delta u\|^2_{L^{2}(\Omega)}+\| u\|^2_{L^{2}(\Omega)}\right)^{1/2}\|u\|_{L^{p}(\Omega)}\left\|e^{(12\pi^2-\varepsilon) u^{2}}\right\|_{L^{\frac{16 \pi^2-\varepsilon} {12\pi^2-\varepsilon}}(\Omega)},\label{com}
\end{align}
where $1 / p+1 / 4+(12\pi^2-\varepsilon) /(16 \pi^2-\varepsilon)=1$. This together with Lemma \ref{best constant} yields $$\sup _{u \in \mathcal{H}\backslash\{0\}} I_{12\pi^2-\varepsilon}(u)<+\infty,$$ for any $\varepsilon>0$. Using the test function $\tilde{U}_k$ constructed in Lemma \ref{best constant} again, one can easily check that for any $\alpha>12\pi^2, I_{\alpha}\left(\tilde{U}_k\right)\rightarrow +\infty$ as $k \rightarrow \infty$.
\end{proof}

Let $\alpha_k$ be an increasing sequence converging to $12\pi^2$, then by the weak compactness of the Banach space $L^{12\pi^2/\alpha_k}$, there exists an exremal function $u_k\in\mathcal{H}\setminus\{0\}$ such that

 \begin{equation*}
\int_{\partial\Omega}\exp(\alpha_k|u_k|^{2})d\sigma=\sup_{u\in\mathcal{H}\setminus\{0\}}\int_{\partial\Omega}\exp(\alpha_k|u|^{2})d\sigma.
\end{equation*}
Furthermore, we can show that the extremal function $u_k\in\mathcal{H}$ is smooth. For this, we first recall the following elliptic regularity result.
\begin{lemma}[\cite{Troianiello},Theorem 3.17]\label{regu}
Suppose that $f\in L^{p}(\Omega)$ and $h\in W^{1,p}(\Omega)$ for some $p\geq2$. Let $u\in W^{1,2}(\Omega)$ be a solution of
\begin{equation*}
\begin{cases}\Delta u=f, & \text { in }\  \Omega, \\ \frac{\partial u}{\partial \nu}=h, & \text { on } \partial \Omega.\\
\end{cases}
\end{equation*}
Then $u\in W^{2,p}(\Omega)$.
\end{lemma}

\begin{lemma}\label{regular}
For any $\alpha_k<12\pi^2$, the functional $I_{\alpha_k}(u)$ defined in $\mathcal{H}$ admits a smooth maximizer.
\end{lemma}
\begin{proof}
Obviously, there exists $u_k\in\mathcal{H}$ such that
$$I_{\alpha_k}(u_k)=\sup _{u \in \mathcal{H}\setminus\{0\}}I_{\alpha_k}(u).$$
Hence $u_k$ satisfies the Euler-Lagrange equation
\begin{equation}\label{fen}
\begin{cases}
 \Delta^2 u_k=\gamma_k, & \forall x\in \Omega, \\
    \frac{\partial\Delta u_k}{\partial \nu}=\frac{u_k\exp(\alpha_k u^2_k)}{\lambda_k}, & \forall x\in \partial\Omega, \\
    \int_{\Omega}|\Delta u_k|^{2}dx=1,\int_{\Omega}u_kdx=0,\frac{\partial u_k}{\partial \nu}=0, &\forall x\in \partial\Omega,
\end{cases}
\end{equation}
where
\begin{equation}\label{sign}
 \lambda_k=-\int_{\partial\Omega}u^2_k\exp(\alpha_k u^2_k)d\sigma,\gamma_k=\int_{\partial\Omega}\frac{u_k\exp(\alpha_k u^2_k)}{\lambda_k|\Omega|}d\sigma.
\end{equation}
By the Orlicz imbedding (see Lemma 3.4 in \cite{Hang1}), we obtain $\exp(u_k^2)\in L^p(\Omega)$ for any $p>1$. Therefore, $\frac{u_k\exp(\alpha_k u^2_k)}{\lambda_k}\in W^{1,q}(\Omega)$ for any $1<q<2$.  We claim that $u_k\in L^{\infty}(\Omega)$. Indeed, we can rewrite (\ref{fen}) as the following systems
\begin{equation}\label{fen1}
\begin{cases}\Delta u_k=v_k, & \text { in } {\Omega}, \\ \frac{\partial u_k}{\partial \nu} =0, & \text { on } \partial \Omega,\\
\end{cases}
\end{equation}
and \begin{equation}\label{fen2}
\begin{cases}\Delta v_k=\gamma_k, & \text { in } {\Omega}, \\ \frac{\partial v_k}{\partial \nu} =h_k, & \text { on } \partial \Omega,\\
\end{cases}
\end{equation}
where $h_k=\frac{u_k\exp(\alpha_k u^2_k)}{\lambda_k}$.  Applying Lemma \ref{regu} for \eqref{fen2}, we know $v_k\in W^{2,q}(\Omega)$. By the Sobolev Embedding theorem, we get $v_k\in L^{\frac{4q}{4-2q}}(\Omega)$. Using Lemma \ref{regu} again for (\ref{fen1}), we derive that $u_k\in W^{2,\frac{4q}{4-2q}}(\Omega)$. Since $q>1$, we can immediately obtain the claim by the Sobolev embedding theorem.
\vskip0.1cm

From the boundedness of $u_k$, we know that $h_k\in W^{1,2}(\Omega)$. Thus we have $h_k\in W^{2,2}(\Omega)$ from Lemma \ref{regu}, hence $h_k\in W^{1,p}(\Omega)$ for some $p>2$. By Lemma \ref{regu} again we have $u_k\in W^{2,p}(\Omega)$, which implies that $u_k\in C^{1}(\Omega)$. Repeating the above procedure, we can conclude that $u_k\in C^{\infty}(\Omega)$.

\end{proof}

Now, we give the following important observation.

\begin{lemma}\label{lamb-gama}It holds $$-\liminf _{k \rightarrow \infty} \lambda_{k}>0, $$ and $|\gamma_k|<c $ for some $c>0$.
\end{lemma}
\begin{proof}
By the element inequality $te^t>e^t-1$, we have

\begin{equation}\label{cont}
 |\partial \Omega|<\sup _{u \in\mathcal{H}\setminus\{0\}} \int_{\partial \Omega} e^{12\pi^2 u^{2}}=\lim _{k \rightarrow \infty} \int_{\partial \Omega} e^{\alpha_k u_{k}^{2}} \leq |\partial \Omega|-\liminf_{k \rightarrow \infty} 12\pi^2\lambda_{k}.
\end{equation}
This implies  $-\liminf _{k \rightarrow \infty} \lambda_{k}>0.$
By (\ref{sign}), (\ref{cont}) and  H{\"o}lder's inequality, we derive
 \begin{align*}
   |\gamma_k|  \leq   \frac{-1}{\lambda_k|\Omega|}\left(\int_{\partial\Omega}u^2_k\exp(\alpha_k u^2_k)d\sigma\right)^{1/2}\left(\int_{\partial\Omega}\exp(\alpha_k u^2_k)d\sigma\right)^{1/2}  \leq c,
 \end{align*}
and the proof is finished.
\end{proof}

 Set $c_{k}=\left|u_{k}\left(x_{k}\right)\right|=\max _{x \in \Omega}\left|u_{k}(x)\right|$. If $\left\{c_{k}\right\}$ is bounded, then by the  elliptic estimates with respect to equation (\ref{fen}), there exists $u \in \mathcal{H} \bigcap C^{\infty}(\Omega)$ such that $u_{k} \rightarrow u$ in $C^{\infty}(\Omega)$ as $k \rightarrow \infty$, and Theorem \ref{mainre} follows immediately. In the sequel,  we assume $c_{k} \rightarrow+\infty$ as $k \rightarrow \infty$. Passing to a subsequence, we may assume that $c_k\geq0$ for all $k$, for otherwise we consider $-u_k$ instead of $u_k$.

\section{Blow-up analysis}
In this section, we consider the blow-up case, that is $c_k\rightarrow\infty$  as $k\rightarrow\infty$. Applying the Adams' inequality in \cite{A}, we know that passing to a subsequence, $x_k\rightarrow p$ for some
$p\in \partial\Omega$. Now, we show that the weak limit of  $u_k$  in $W^{2,2}(\Omega)$ is zero. Furthermore, $u_k$ must concentrate around the blow-up point $p$.
\begin{lemma}\label{concen}
If $c_{k} \rightarrow+\infty$, then $u_{k} \rightharpoonup 0$ in $W^{2,2}(\Omega)$ and $u_{k} \rightarrow 0$ in $L^{p}(\Omega)$ for any $1\leq p< \infty$. Moreover, we have \\

(i) $\left|\Delta u_{k}\right|^{2}dx \rightarrow \delta_{p}$ in the sense of measures;\\

(ii) $e^{\alpha_{k} u_{k}^{2}}$ is bounded in $L^{p}\left(\Omega \backslash B_{\delta}\left(p\right)\right)$, for any $p \geq 1, \delta>0$;\\

(iii) $u_{k} \rightarrow 0$ in $C^{3, \gamma}\left(\Omega \backslash B_{\delta}\left(p\right)\right)$, for any $\gamma \in(0,1), \delta>0.$
\end{lemma}
\begin{proof}Since $u_{k}$ is bounded in $W^{2,2}(\Omega)$, we assume that $u_{k} \rightharpoonup u_{0}$ in $W^{2,2}(\Omega)$ with some  $u_{0} \in W^{2,2}(\Omega)$. The compactness of the embedding of $W^{2,2}(\Omega)$ into $L^{p}(\Omega)$ implies $u_{k} \rightarrow u_{0}$ in $L^{p}(\Omega)$ for any $p \geq 1$. If $u_{0} \neq 0$, then by the concentration compactness principle (see Proposition 3.2 of \cite{Hang1}),   $ e^{16\pi^2 u_{k}^{2}}$ is bounded in $L^{p}(\Omega)$ for some $p>1$. Similar as (\ref{com}), we can find some $\varepsilon_0>0$ such that $ e^{12\pi^2 u_{k}^{2}}$ is bounded in $L^{1+\varepsilon_0}(\partial\Omega)$,  hence $\frac{\partial\Delta u_{k}}{\partial\nu}$ is bounded in $L^{1+\varepsilon_0}(\partial\Omega)$. Using the same argument in Lemma \ref{regular}, we get $u_{k}\in L^{\infty}(\Omega)$. This contradicts with $c_{k} \rightarrow+\infty$. Hence, we have $u_{0}=0$.
\vskip0.1cm

Now, we show that $u_k$ must concentrate around the blow-up point $p$. Let \begin{equation*}
A=\left\{q \in \Omega: \lim _{r \rightarrow 0} \liminf _{k \rightarrow \infty} \int_{B_{r}(q)}\left|\Delta u_{k}\right|^{2}dx>0\right\}.
\end{equation*}
We claim that $A$ contains only one point. Suppose that the claim does not hold. Then, for any $q \in \Omega$, we have $\lim\limits _{r \rightarrow 0} \liminf _{k \rightarrow \infty} \int_{B_{r}(q)}\left|\Delta u_{k}\right|^{2}dx<1$. Then there exist positive numbers $r$ and $\delta$ such that
$$
\int_{B_{r}(q)}\left|\Delta u_{k}\right|^{2}dx \leq \delta(q)<1.
$$
Using the same argument as that in (\ref{com}) again, we see that there exists a constant $\alpha(q)>12\pi^2$ such that
$$
\int_{\partial \Omega \bigcap B_{r}(q)} e^{\alpha(q) u_{k}^{2}}d\sigma \leq C_{q},
$$
for some constant $C_{q}$ depending on $q$. Hence, there exists an $\alpha>12\pi^2$ such that
$$
\int_{\partial \Omega} e^{\alpha u_{k}^{2}}d\sigma \leq C,
$$
by using the covering argument. Therefore, it follows from the Vitali convergence lemma that
$$
\lim_k\int_{\partial \Omega} e^{\alpha_k u_{k}^{2}}d\sigma=|\partial \Omega|,
$$
 which is impossible by the  choice of $u_k$.
Next we show that $A=\{p\}$ and $$\lim _{r \rightarrow 0} \liminf _{k \rightarrow \infty} \int_{B_{r}(p)}\left|\Delta u_{k}\right|^{2}=1.$$ Suppose not, repeating the argument above, we can obtain $u_{k}\in L^{\infty}(B_{\delta}(p))$ for some $\delta>0$, which contradicts with $c_{k} \rightarrow+\infty$, and the statement (i) is proved.
\vskip0.1cm

The statement (ii) follows from (i) and Lemma \ref{subcritical}, and the statement (iii) can be proved by the standard regularity argument, we omit the details.
\end{proof}

To understand the asymptotic behavior of $u_k$ near the blow-up point $p$, we define $$r^3_k=-\frac{\lambda_k}{c^2_k}\exp\left(-\alpha_kc^2_k\right).$$  Indeed, $r_k$ decay very fast as $k\rightarrow\infty$, that is,

\begin{lemma} For any $\gamma<12\pi^2$, there holds $e^{-\gamma c_k^2}r^3_k\rightarrow0$ as $k\rightarrow\infty$.
\end{lemma}
\begin{proof} For any $\gamma<12\pi^2$, we have
\begin{align*}
  c^2_kr^3_ke^{\gamma c_k^2}&=e^{(\gamma-\alpha_k)c_k^2} \int_{\partial \Omega} u_{k}^{2} e^{\alpha_k u_{k}^{2}}d\sigma \leq\int_{\partial \Omega} u_{k}^{2} e^{\alpha_k u_{k}^{2}}e^{(\gamma-\alpha_k)u_k^2}d\sigma \\
  & \leq  \int_{\partial \Omega} u_{k}^{2} e^{\gamma u_{k}^{2}}d\sigma \leq  \left(\int_{\partial \Omega} u_{k}^{s}d\sigma\right)^{\frac{2}{s}}  \left(\int_{\partial \Omega}e^{\frac{\gamma s}{s-2} u_{k}^{2}}d\sigma\right)^{\frac{s-2}{s}} \\
 & \leq c,
\end{align*}
provided $s$ large enough, where we have used the subcritical trace Adams inequality.\end{proof}

The following important observation allows us to choose $\{x_k\}$ on the boundary $\partial\Omega$.
\begin{lemma}\label{location}
There exists some point $\tilde{x}_k\in\partial\Omega$ such that \begin{equation*}
|u_k(\tilde{x}_k)-u_k(x_k)|=o_k(1),\end{equation*}as $k\rightarrow\infty$.
\end{lemma}
\begin{proof}
If $x_k\in \partial\Omega$, we take a smooth vector field $\vec{\nu}(x)$ whose restriction on $\partial\Omega$ is the outward unit normal vector field, and let $\gamma_k(t): [0,\infty)\mapsto \Omega$ be the integral curve of $\vec{\nu}(x)$  such that $\gamma_k(0)=x_k, \gamma_k(1)=\tilde{x}_k\in \partial\Omega$ and $\frac{\partial \gamma_k(t)}{\partial t}=\vec{\nu}(\gamma_k(t))$. We can obtain  $|x_k-\gamma_k(1)|\rightarrow0$ from the fact $dist(x_k,\partial\Omega)\rightarrow0$, as $k\rightarrow\infty$. Since $u_k\in C^{\infty}(\Omega)$ satisfies $\frac{\partial u_k}{\partial \nu}\big|_{\partial\Omega}=0$, then $\frac{\partial u_k(\gamma(t))}{\partial t}=o_k(1)$ as $k\rightarrow \infty$. Therefore, by the mean value theorem, there exists some $\xi\in(0,1)$ such that
$$|u_k(x_k)-u_k(\gamma(1))|=\frac{\partial u_k(\gamma(t))}{\partial t}|_{t=\xi}\cdot|x_k-\gamma_k(1)|=o_k(1),$$and the proof is finished. \end{proof}

In view of Lemma \ref{location}, we can take $x_k=\tilde{x}_k\in \partial\Omega$ and then \begin{equation}\label{rechoose}u_k(x_k)=c_k+o_k(1),\end{equation} as $k\rightarrow\infty$. Define two
sequences of functions on $\partial\Omega$, namely,
\begin{equation*}
\begin{cases}
 &\phi_{k}(x)=\frac{u_{k}\left(x_{k}+r_{k} x\right)}{c_{k}}, \ x \in \Omega_k=\{x: x_k+r_kx \in \Omega\}\\
&\psi_{k}(x)=c_{k}\left(u_{k}\left(x_{k}+r_{k} x\right)-c_{k}\right),\ x\in \Omega_k.
\end{cases}
\end{equation*}
Up to translation and rotation, we can easily obtain $\Omega_k\rightarrow \mathbb{R}^{4}_{+}$ as $k\rightarrow +\infty$.

\begin{lemma}
$\phi_{k}(x)\rightarrow 1 \text{ in } C^3_{loc}(\overline {\mathbb{R}_{\rm{ + }}^4})$.
\end{lemma}
\begin{proof}
By (\ref{fen}), for $k$ large enough we have
\begin{equation}\label{maineq}
\begin{cases}
 \Delta^2 \phi_{k}=\frac{r^4_k}{c_k}\gamma_k, & \forall x\in B_R(0)\bigcap \Omega_k, \\
    \frac{\partial}{\partial \nu}\Delta \phi_{k}=\frac{r_k^3u_k\exp(\alpha_k u^2_k)}{c_k\lambda_k}, & \forall x\in B_R(0)\bigcap\partial \Omega_k,
\end{cases}
\end{equation}
for any $R>0$. By the definition of $r_k$, we have $$|\frac{r^4_k}{c_k}\gamma_k|=\frac{\lambda_k}{c^2_k}\exp\left(-\alpha_kc^2_k\right)
   \frac{r_k}{c_k}\int_{\partial\Omega}\frac{u_k\exp(\alpha_k u^2_k)}{\lambda_k|\Omega|}d\sigma\leq \frac{|\partial \Omega|}{|\Omega|}\frac{r_k}{c^2_k}\rightarrow0$$
   and
$$\left|\frac{r_k^3u_k\exp(\alpha_k u^2_k)}{c_k\lambda_k}\right|\leq \frac{1}{c^3_k}\rightarrow0,$$
as $k\rightarrow\infty$. Since $\phi_{k}$ is  bounded in $L_{loc}^1(\overline {B_{R}(0)\bigcap \Omega_k})$ and $\phi_{k}(x_k)=1+o_k(1)$, by the standard elliptic regularity argument, we have $\phi_{k}\rightarrow 1$ in $C^3_{loc}(\overline{B_{R/2}(0)\bigcap \Omega_k}$).
\end{proof}

In order to obtain the limit behavior of $\psi_k$, we need to check the following growth condition:
\begin{lemma}\label{lem5}
$$\int_{B_{R}(0)\bigcap \Omega_k}|\Delta\psi_k|dx\leq CR^2.$$
\end{lemma}
\begin{proof}
Direct computation gives that $\int_{B_{R}(0)\bigcap \Omega_k}|\Delta\psi_k|dx=c_kr_k^{-2}\int_{B_{Rr_k}(x_k)\bigcap \Omega}|\Delta u_k|dx$.
Since $\frac{u_k(r_kx+x_k)}{c_k}\rightarrow 1$ in $C^{3}(B_{R}\bigcap \Omega_k)$ for any $R>0$, in order to prove this lemma, we only need to show that
$$(Rr_k)^{-2}\int_{B_{Rr_k}(x_k)\bigcap \Omega}|u_k\Delta u_k|dx\lesssim 1.$$
Applying H{\"o}lder's inequality in Lorentz space (see \cite{Neil}), we get
\begin{equation}\begin{split}
(Rr_k)^{-2}\int_{B_{Rr_k}(x_k)\bigcap \Omega}|u_k\Delta u_k|dx&\leq (Rr_k)^{-2}\|\chi_{B_{Rr_k}(x_k)\bigcap \Omega_k}\|_{L^{2,1}(B_{Rr_k}(x_k)\bigcap \Omega)}\|u_k\Delta u_k\|_{L^{2,\infty}(B_{Rr_k}(x_k)\bigcap \Omega)}\\
&\lesssim \|u_k\Delta u_k\|_{L^{2,\infty}(B_{Rr_k}(x_k))}.
\end{split}\end{equation}
Now, we start to prove that $\|u_k\Delta u_k\|_{L^{2,\infty}(\Omega)}\lesssim 1$.
Let $G$ denote the Green function of Laplacian operator with Neuman Boundary condition:
\begin{equation*}
\begin{cases}
 &-\Delta G_x(y)=\delta_x(y)-\frac{1}{|\Omega|},\ \ x,\ y \in \overline{\Omega},\\
&\frac{\partial G}{\partial \nu}|_{\partial\Omega}=0, \\
& \int_{\Omega}G_x(y)dy=0, \ \ x\in \overline{\Omega}.
\end{cases}
\end{equation*}
Obviously $G(x,y)$ satisfies $G(x,y)\lesssim |x-y|^{-2}$ for any $x$, $y\in \Omega$.
By integration by parts together with $\int_{\Omega}u_k(x)dx=0$ and $\frac{\partial u_k}{\partial \nu}|_{\partial\Omega}=0$, we derive that
$$|u_k(x)|\lesssim \int_{\Omega}|\Delta u_k||x-y|^{-2}dy$$ and
$$|\Delta u_k(x)|\lesssim \int_{\partial\Omega}|x-y|^{-2}f_k(y)d\sigma_y+\int_{\Omega}|x-y|^{-2}\gamma_kdy+\frac{1}{|\Omega|}\int_{\Omega}|\Delta u_k|dy\lesssim 1+\int_{\partial\Omega}|x-y|^{-2}f_k(y)d\sigma_y,$$
where $f_k=\frac{u_k\exp(\alpha_k u^2_k)}{\lambda_k}$. Then it follows that
\begin{equation}\begin{split}
|u_k(x)||\Delta u_k(x)|\lesssim \Big(\int_{\Omega}|(\Delta u_k)(y)||x-y|^{-2}dy \Big)\Big(1+\int_{\partial\Omega}|x-z|^{-2}f_k(z)d\sigma_z\Big).
\end{split}\end{equation}
Now, we claim that $$\big\|\big(\int_{\Omega}|\Delta u_k(y)||x-y|^{-2}dy\big) \big(1+\int_{\partial\Omega}|x-z|^{-2}f_k(z)d\sigma_z\big)\big\|_{L^{2,\infty}}\lesssim 1.$$
It suffices to prove that  $$\big\|\big(\int_{\Omega}|\Delta u_k(y)||x-y|^{-2}dy\big)\big(\int_{\partial\Omega}|x-z|^{-2}f_k(z)d\sigma_z\big)\big\|_{L^{2,\infty}}\lesssim 1.$$
For any $\epsilon>0$ sufficiently small, using the following estimate (see \cite{Sc1})
$$|x-y|^{-2}|x-z|^{-2}\leq |x-y|^{-2-\epsilon}|x-z|^{-2+\epsilon}+|z-y|^{-2}|x-z|^{-2},$$ we obtain
\begin{equation}\begin{split}
&\big\|\big(\int_{\Omega}|\Delta u_k(y)||x-y|^{-2}dy\big)\big(\int_{\partial\Omega}|x-z|^{-2}f_k(z)d\sigma_z\big)\big\|_{L^{2,\infty}}\\
&\ \ \leq \big\|\big(\int_{\Omega}|\Delta u_k(y)||x-y|^{-2-\epsilon}dy\big)\big(\int_{\partial\Omega}|x-z|^{-2+\epsilon}f_k(z)d\sigma_z\big)\big\|_{L^{2,\infty}}\\
&\ \ \ \ +\big\|\int_{\partial\Omega}\big(\int_{\Omega}|\Delta u_k(y)||z-y|^{-2}dy\big)f_k(z)|x-z|^{-2}d\sigma_z\big\|_{L^{2,\infty}}\\
&\ \ :=I_1+I_2.
\end{split}\end{equation}
Applying the generalized H\"{o}lder's inequality involving the Lorentz norm, we derive that
\begin{equation*}\begin{split}
I_1&\leq \big\|\int_{\Omega}|\Delta u_k(y)||x-y|^{-2-\epsilon}dy\big\|_{L^{\frac{4}{\epsilon}}(\Omega)}\|\int_{\partial\Omega}|x-z|^{-2+\epsilon}f_k(z)d\sigma_z\|_{L^{\frac{4}{2-\epsilon},\infty}(\Omega)}\\
&:= I_{11}\times I_{12}.
\end{split}\end{equation*}
For $I_{11}$, the boundedness of fractional integral operator directly gives $I_{11}\lesssim \|\Delta u_k\|_{L^2(\Omega)}$. For $I_{12}$, we claim that it can be dominated by $\|f_k\|_{L^1(\partial\Omega)}$. Define the auxiliary integral operators
$$T_{\epsilon, r}^1(x)=\int_{\{\partial\Omega \cap |x-y|<r\}}\frac{f_k(y)}{|x-y|^{2-\epsilon}}d\sigma_y,\ \ T_{\epsilon, r}^2(x)=\int_{\{\partial\Omega \cap |x-y|\geq r\}}\frac{f_k(y)}{|x-y|^{2-\epsilon}}d\sigma_y.$$
Obviously, $$\int_{\Omega}|T_{\epsilon, r}^1|dx\leq \big(\sup_{y\in \partial\Omega}\int_{\{|x-y|<r\}}\frac{1}{|x-y|^{2-\epsilon}}dx\big) \|f_k\|_{L^1(\partial\Omega)}\lesssim r^{2+\epsilon}\|f_k\|_{L^1(\partial\Omega)}$$
and $$\|T_{\epsilon, r}^2\|_{L^{\infty}(\Omega)}\leq \frac{1}{r^{2-\epsilon}}\|f_k\|_{L^1(\partial\Omega)}.$$
For any $\lambda>0$, we can write
$$|\{x:T_{\epsilon, r}^1+T_{\epsilon, r}^1>2\lambda\}|\leq |\{x:T_{\epsilon, r}^1>\lambda\}|+|\{x:T_{\epsilon, r}^2>\lambda\}|.$$
Choosing $r$ such that $\frac{1}{r^{2-\epsilon}}\|f_k\|_{L^1(\partial\Omega)}=\lambda$, then $|\{x:T_{\epsilon, r}^2>\lambda\}|=0$.
Hence, we deduce that $$|\{x:T_{\epsilon, r}^1+T_{\epsilon, r}^1>2\lambda\}|\leq \frac{r^{2+\epsilon}}{\lambda}\|f_k\|_{L^1(\partial\Omega)}=
\frac{1}{\lambda^{\frac{4}{2-\epsilon}}}\|f_k\|_{L^1(\partial\Omega)}^{\frac{4}{2-\epsilon}},$$
which gives that $I_{12}\lesssim \|f_k\|_{L^1(\partial\Omega)}$, and the claim is proved.

Gathering the estimates of $I_{11}$ and $I_{12}$,
we derive that $I_1\lesssim  \|\Delta u_k\|_{L^2(\Omega)}\|f_k\|_{L^1(\partial\Omega)}$. For $I_2$, obviously
\begin{equation*}\begin{split}
I_2&\lesssim \big\|\int_{\Omega}|\Delta u_k(y)||z-y|^{-2}f_k(z)dy\big\|_{L^1(\partial\Omega)}\lesssim \big\|\Delta u_k\big\|_{L^2(\Omega)}\|\int_{\partial\Omega}|z-y|^{-2}f_k(z)d\sigma_z\|_{L^2(\Omega)}.\\
\end{split}\end{equation*}
According to Corollary 6.16 in \cite{Ben}, we derive that
$$\big\|\int_{\partial\Omega}|z-y|^{-2}f_k(z)d\sigma_z\big\|_{L^2(\Omega)} \lesssim \int_{\partial\Omega}f_k(z)\log^{\frac{1}{2}}(1+f_k(z))d\sigma_z.$$
Since $f_k=\frac{u_k\exp(\alpha_k u^2_k)}{\lambda_k}$, it is easy to check that $\int_{\partial\Omega}f_k(z)\log^{\frac{1}{2}}(1+f_k(z))d\sigma_z\lesssim 1$.
Combining the estimates of $I_1$ and $I_2$, we derive that $$\big\|\big(\int_{\Omega}|\Delta u_k(y)||x-y|^{-2}dy\big)\big(\int_{\partial\Omega}|x-z|^{-2}f_k(z)d\sigma_z\big)\big\|_{L^{2,\infty}}\lesssim 1,$$
which accomplishes the proof of Lemma \ref{lem5}.

\end{proof}

\begin{lemma}\label{limiteq}
We have $\psi_{k}(x)\rightarrow \psi(x',t) \text{ in } C^3_{loc}( {\overline{B^{+}_{R}(0)}}) $ ($x'\in \partial \mathbb{R}_{+}^{4}, t\in \mathbb{R}^{+} $),  where $\psi(x',t)$ satisfies the equation
\begin{equation*}
\begin{cases}
\Delta^2\psi=0\  & x\in \mathbb{R}_{+}^{4},\\
\frac{\partial\Delta\psi}{\partial t}=\exp(24\pi^2 \psi)\ & x\in \partial\mathbb{R}_{+}^{4}, \\
\psi(0)=\sup \psi=0,\\
 \frac{\partial\psi}{\partial t}=0, \ \ & x\in \partial\mathbb{R}_{+}^{4}.
\end{cases}
\end{equation*}
Furthermore, $\psi$ must take the form as
$$\psi=-\frac{1}{8\pi^2}\log\left( \left(1+\left({\frac{\pi}{2}}\right)^{\frac{2}{3}}t   \right)^2+  \left({\frac{\pi}{2}}\right)^{\frac{4}{3}}|x'|^2   \right)+\frac{1}{2^{\frac{8}{3}}\pi^{\frac{4}{3}}}\frac{t}{  \left(1+\left({\frac{\pi}{2}}\right)^{\frac{2}{3}}t  \right)^2+  \left({\frac{\pi}{2}}\right)^{\frac{4}{3}}|x'|^2   }.$$
\end{lemma}
\begin{proof}
By (\ref{fen}), we can easily obtain
\begin{equation}\label{main}
\begin{cases}
 \Delta^2 \psi_{k}={c_kr^4_k}\gamma_k, & \forall x\in B_R(0)\bigcap \Omega_k, \\
  \frac{\partial\psi_k}{\partial t}=0, & \forall x\in B_R(0)\bigcap\partial \Omega_k,\\
    \frac{\partial\Delta \psi_{k}}{\partial t}=\frac{u_k\exp(\alpha_k \psi_k(1+\frac{u_k}{c_k}))}{c_k}, & \forall x\in B_R(0)\bigcap\partial \Omega_k,
\end{cases}
\end{equation}
for any $R>0$. Let $-\Delta \psi_k=v_k$, then $\psi_k$ and $v_k$ satisfy the following equation respectively:
\begin{equation}\begin{cases}
-\Delta \psi_k=v_k, & \forall x\in B_R(0)\cap \Omega_k,\\
 \frac{\partial\psi_k}{\partial t}=0,& \forall x\in B_R(0)\bigcap\partial \Omega_k;
\end{cases}\end{equation}
and
\begin{equation}\begin{cases}
-\Delta v_k={c_kr^4_k}\gamma_k, & \forall x\in B_R(0)\cap \Omega_k,\\
 \frac{\partial v_k}{\partial t}=\frac{u_k\exp(\alpha_k \psi_k(1+\frac{u_k}{c_k}))}{c_k}, & \forall x\in B_R(0)\bigcap\partial \Omega_k.
\end{cases}\end{equation}
Noticing $$\frac{\partial v_k}{\partial t}=\frac{u_k\exp(\alpha_k \psi_k(1+\frac{u_k}{c_k}))}{c_k}\in L^{\infty}(B_R(0)\cap\partial \Omega_k),$$ applying Lemma \ref{lem5} and the standard elliptic regularity, we deduce that
$$\|v_k\|_{C^{1,\alpha}(\overline {B_{R/2}(0)\cap \Omega_k})}\lesssim 1.$$
 Then there exists some $v\in C^{1,\alpha}(\overline {B_{R/2}(0)\cap \Omega_k})$ such that
$v_k\rightarrow v$ in $C^{1,\beta}(\overline {B_{R/2}(0)\cap \Omega_k})$ for any $\beta<\alpha$. Let $\widetilde{\psi_k(x)}$ be the even extension of
$\psi_k$ with respect to the boundary $\partial B^{+}_R(0)\cap\partial \mathbb{R}^4_{+}$, then we have $-\Delta \widetilde{\psi_k}\in C^{1, \alpha}(B_{R}(0)\cap \Omega_k)$, $\widetilde{\psi_k(x)}\leq \psi_k(0)=0$. Using the Harnack inequality and elliptic regularity estimates, we get
$\|\widetilde{\psi_k}\|_{C^{3,\alpha}(B_{R}(0)\cap \Omega_k)}\lesssim C$. Hence there exists $\psi \in C^{3,\beta} (\overline {B_{R}(0)\cap \Omega_k}) $ such that
$\psi_k\rightarrow \psi $ in $C^{3,\beta}(\overline {B_{R}(0)\cap \Omega_k})$ for any $\beta<\alpha$, where $\psi$ satisfies the equation
\begin{equation*}
\begin{cases}
 &\Delta^2\psi=0 \text{ in } \mathbb{R}_{+}^{4},\\
&\frac{\partial\Delta\psi}{\partial t}=\exp(24\pi^2 \psi) \text{ on } \partial\mathbb{R}_{+}^{4}, \\
&\psi(0)=\sup \psi=0,\\
& \frac{\partial\psi}{\partial t}=0, \text{ on } \partial\mathbb{R}_{+}^{4}.
\end{cases}
\end{equation*}

From (\ref{rechoose}), it is not difficult to see that \begin{equation}\int_{B_R\bigcap \partial \mathbb{R}^4_{+}}\exp\left(24\pi^2\psi\right)\leq-\int_{B_{Rr_k}\bigcap \partial\Omega}\frac{u_k^2\exp\left(\alpha_ku^2_k\right)}{\lambda_k}\leq1.\end{equation}
Next, we will prove that $\psi$ must take the form as
$$\psi=-\frac{1}{8\pi^2}\log\left( \left(1+\left({\frac{\pi}{2}}\right)^{\frac{2}{3}}t   \right)^2+  \left({\frac{\pi}{2}}\right)^{\frac{4}{3}}|x'|^2   \right)+\frac{1}{2^{\frac{8}{3}}\pi^{\frac{4}{3}}}\frac{t}{  \left(1+\left({\frac{\pi}{2}}\right)^{\frac{2}{3}}t  \right)^2+  \left({\frac{\pi}{2}}\right)^{\frac{4}{3}}|x'|^2   }.$$
Indeed, let $\phi(x)=\int_{\partial \mathbb{R}^{+}_{4}}P(x,y')\psi(y', 0)dy'$, where $x=(x',t)$, $y=(y',t)$ and $P(x,y')=\frac{4}{\pi^2}\frac{t^3}{|x-y'|^6}$ is the Poisson kernel for the bi-laplace operator on the upper half space. It is not difficult to check that $\phi$ satisfies the following equation
\begin{equation}\begin{cases}
(-\Delta)^2\phi=0,\ & x\in \mathbb{R}^{4}_{+},\\
\phi=\psi(x),\ & x\in \partial \mathbb{R}^{4}_{+},\\
\frac{\partial \phi}{\partial t}=0, \  & x\in \partial \mathbb{R}^{4}_{+};\\
\end{cases}\end{equation}
and $\int_{B_R^{+}(0)}|\Delta \phi|dx\leq CR^2$.
\vskip0.1cm

Let $w=\psi-\phi$, then $w$ satisfies

\begin{equation}\begin{cases}
(-\Delta)^2w=0,\ & x\in\mathbb{ R}^{4}_{+},\\
w=0, \  & x\in \partial \mathbb{R}^{4}_{+},\\
\frac{ \partial w}{\partial t}=0,\ & x\in \partial \mathbb{R}^{4}_{+}.\\
\end{cases}\end{equation}
Noticing $\int_{B_R^{+}(0)}|\Delta w|dx\leq \int_{B_R^{+}(0)}|\Delta \psi|dx+\int_{B_R^{+}(0)}|\Delta \phi|dx\leq CR^2$, one can deduce that $w$ must be equal to zero. Hence $\psi(x)=\int_{\partial \mathbb{R}^{+}_{4}}P(x,y')\psi(y', 0)d\xi$.
Set $\psi_0(x')=\psi(x',0)$, then we know that $\frac{1}{2}\frac{\partial \Delta \psi}{\partial t}|_{\partial \mathbb{R}^{4}_{+}}=(-\Delta)^{\frac{3}{2}}\psi_0$ and $\psi_0(x')$ satisfies the following equation in the distributional sense:
\begin{equation}\begin{cases}
(-\Delta)^{\frac{3}{2}}\psi_0=\frac{1}{2}e^{24\pi^2\psi_0},  x'\in \partial\mathbb{ R}^{3}\\
\int_{\mathbb{ R}^3}e^{24\pi^2\psi_0(x')}dx'\leq \frac{1}{2}.\\
\end{cases}\end{equation}
Let $\eta_0(x')=8\pi^2\psi_0(x')+\frac{1}{3}\log(2\pi^2)$, then $\eta_0$ satisfies
\begin{equation}\begin{cases}
(-\Delta)^{\frac{3}{2}}\eta_0=2e^{3\eta_0},\ x'\in \mathbb{ R}^3,\\
\int_{\mathbb{ R}^3}e^{3\eta_0}dx'\leq \pi^2.\\
\end{cases}\end{equation}
\vskip0.1cm

From the result of Hyder (\cite{Hyder}), we known that $\eta_0(x')$ can be decomposed as $\eta_0=v+p$, where $p$ is a polynomial of degree at most $2$ and $v(x')=-\alpha \log(|x'|)+o(\log |x'|)$ as $|x'|\rightarrow +\infty$. Furthermore, $\eta_0(x')=\log{\frac{2\lambda}{1+\lambda^2|x'-x'_0|^2}}$ if and only if $p$ is a constant. Noticing that  $\psi(x',t)$ is a Poisson extension of $\psi_0$ on $\mathbb{R}^{4}_{+}$ and $\int_{B_{R}^{+}}|\Delta \psi|dx\leq CR^2$, then we deduce that $p$ must be equal to constant. This proves $\eta_0(x')=\log{\frac{2\lambda}{1+\lambda^2|x'-x'_0|^2}}$. Since $\psi(x)\leq \psi(0)=\sup_{x\in \mathbb{R}^{4}_{+}}\psi(x)=0$, hence $\psi$ must take the form as
$$\psi(x',t)=-\frac{1}{8\pi^2}\log\left( \left(1+\left({\frac{\pi}{2}}\right)^{\frac{2}{3}}t   \right)^2+  \left({\frac{\pi}{2}}\right)^{\frac{4}{3}}|x'|^2   \right)+\frac{1}{2^{\frac{8}{3}}\pi^{\frac{4}{3}}}\frac{t}{  \left(1+\left({\frac{\pi}{2}}\right)^{\frac{2}{3}}t  \right)^2+  \left({\frac{\pi}{2}}\right)^{\frac{4}{3}}|x'|^2   },$$
where the second term ensures $\frac{\partial\psi}{\partial t}\big|_{\partial \mathbb{R}^4_+}=0$. By an easy computation, one can see that \begin{equation}\label{equto1}\int_{\partial \mathbb{R}^4_{+}}\exp\left(24\pi^2\psi(x')\right)dx'=1.\end{equation}

\end{proof}

\subsection{Polyharmonic truncation functions}
We first introduce some notations. If $x_0\in\partial \Omega$, for small $\delta>0$, denote by $M_{\delta,x_0}=B_\delta(x_0)\bigcap \overline{\Omega}$. We can choose a Fermi coordinate  (see \cite{MM}) for
$M_{\delta,x_0}$ by the map $
\theta: M_{\delta} \rightarrow \overline{{B}_{\delta}^{3}(0)}\times[0, \delta]
$, where $\theta(x_0)=0$. We will identify $M_{\delta,x_0}$ with $\overline{{B}_{\delta}^{3}(0)} \times$ $[0, \delta]$ through the map $\theta$.
Under the Fermi coordinate, we can write the metric on the $M_{\delta,p}$ by $g=g_{i j} d x_{i} \otimes d x_{j}+d t \otimes d t(i,j\in \,\{1,2,3\})$, where $(1-\epsilon)\delta_{i,j}\leq g_{ij}\leq (1+\epsilon)\delta_{i,j}$ for small $\epsilon>0$.
\vskip0.1cm

We choose a Fermi coordinate system $(U_k,\theta_k)$ near the point $x_k$  such that $\theta_k(x_k)=0$, and $\theta_k(U_k\cap \Omega)\subseteq \mathbb{R}_+^4=\{x=(x',t)\in \mathbb{R}^4:t>0\}$, and $\theta_k(U_k\bigcap\partial\Omega)\subseteq \partial \mathbb{R}_+^4$. In the following, we make an even extension for $u_k\circ\theta_k^{-1}$ in the direction of $t$ under the Fermi coordinate system $(U_k,\theta_k)$:
\begin{equation*}
\begin{cases}
 &\tilde{u}_{k}(x)= u_k\circ\theta_k^{-1}(x',t), \text{if } t\geq0, \\
 &\tilde{u}_{k}(x)= u_k\circ\theta_k^{-1}(x',-t), \text{if } t<0,
\end{cases}
\end{equation*}
then $\tilde{u}_{k}(x)\in W^{2,2}(B_r(0))$ with $\parallel \Delta\tilde{u}_{k}\parallel_{L^2(B_r(0))}=2\parallel \Delta \tilde{u}_{k}\parallel_{L^2(B_r^{+}(0))}$ for  small $r>0$.
\vskip 0.1cm

Now, we  need some bi-harmonic truncation
functions $\tilde{u}_{k}^{M}$ which was studied in \cite{DelaTorre}. Roughly speaking, the
value of the truncations functions $\tilde{u}_{k}^{M}$ are close to $\frac{c_{k}}{M}$
in a small neighborhood of $0$, and coincides with $\tilde{u}_{k}$ outside the
same neighborhood.

\begin{lemma}
\label{truncation}\cite[Lemma 4.20]{DelaTorre}For any $M>1$ and $k\in%
\mathbb{N}
$, there exists a radius $\tilde{\rho}_{k}^{M}>0$ and a constant $c=c\left(  M\right)
$ such that

1. $\tilde{ u}_{k}\geq\frac{c_{k}}{M}$ in $B_{\tilde{\rho}_{k}^{M}}\left( 0\right)  ;$

2.$\left\vert \tilde{u}_{k}-\frac{c_{k}}{M}\right\vert \leq\frac{c}{c_{k}}$ on
$\partial B_{\tilde{\rho}_{k}^{M}}\left(  0\right)   ;$

3.$\left\vert \nabla^{l}\tilde{u}_{k}\right\vert \leq\frac{c}{c_{k}\left(  \tilde{\rho}
_{k}^{M}\right)  ^{l}}$ on $\partial B_{\tilde{\rho}_{k}^{M}}\left(  0\right)  $
for any $1\leq l\leq3;$

4. $\tilde{\rho}_{k}^{M}\rightarrow0$, and $\frac{\tilde{\rho}_{k}^{M}}{r_{k}}\rightarrow
+\infty$, as $k\rightarrow\infty.$
\end{lemma}

Let $\tilde{v}_{k}^{M}\in C^{4}\left( (\overline{B_{\tilde{\rho}_{k}^{M}}\left(  0\right)
})\right)$ be the unique solution of
\begin{equation}\label{eqne}
\genfrac{\{}{.}{0pt}{0}{\Delta^{2}(\tilde{v}_{k}^{M})%
=0\ \ \ \ \ \ \ \ \ \ \ \ \ \ \ \ \text{\ in }B_{\tilde{\rho}_{k}^{M}}\left(
0\right) ,}{\partial_{\nu}^{i}(\tilde{v}_{k}^{M})=\partial_{\nu}^{i}(\tilde{u}_{k})   \ \ \ \ \ \ \text{on
}\partial B_{\tilde{\rho}_{k}^{M}}\left( 0\right) ,i=0,1.}%
\end{equation}

We consider the function
\begin{equation}\label{test}
u_{k}^{M}=%
\genfrac{\{}{.}{0pt}{0}{\tilde{v}_{k}^{M}\circ \theta_{k}\ \ \ \ \ \ \ \ \ \text{\ in }\theta_k^{-1}(B_{\tilde{\rho}_{k}^{M}}\left(
0\right)),}{ u_k\ \ \ \ \ \ \ \ \ \ \ \  \ \ \text{in }\Omega\backslash \theta_k^{-1}(B_{\tilde{\rho}_{k}^{M}%
}\left( 0\right))  .}%
\end{equation}

\begin{lemma}
\cite[Lemma 4.21]{DelaTorre}For any $M>1$, we have
\[
u_{k}^{M}=\frac{c_{k}}{M}+O\left(  c_{k}^{-1}\right)  ,
\]
uniformly on $\theta_k^{-1}(\overline{B_{\tilde{\rho}_{k}^{M}}\left(0\right)  })$.
\end{lemma}
\begin{remark}\label{remark}
Using  the explicit form of the Green function of $\Delta^2$ on balls, namely the Boggio's formula,  in \cite{Boggio}) and the representation formula of solutions for \eqref{eqne}, one can see that  $\frac{\partial u_{k}^{M}}{\partial \nu}=0$, for any $x\in \theta_k^{-1}(B_{\tilde{\rho}^{M}_{k}}(0))\bigcap \partial\Omega$.

\end{remark}
\begin{lemma}
\label{cufoff mod}For any $M>1$, there holds
\[
\underset{k\rightarrow\infty}{\lim\sup}\int_{%
\Omega}  \left\vert \Delta u_{k}^{M}\right\vert ^{2}  dx\leq\frac{1}{M}.
\]

\end{lemma}
\begin{proof}Testing (\ref{maineq}) with $\left(  u_{k}-u_{k}^{M}\right)  $, by Lemma
\ref{truncation}, Lemma \ref{lamb-gama} and Remark \ref{remark}, for any $R>0,$ we have
\begin{align*}
&  \int_{\theta_k^{-1}(B_{\tilde{\rho}^{M}_{k}}(0))\bigcap\Omega}  \Delta u_{k}\Delta\left(  u_{k}-u_{k}%
^{M}\right)   dx\\
&  = \int_{\theta_k^{-1}(B_{\tilde{\rho}^{M}_{k}}(0))\bigcap\Omega}  \gamma_k \left(  u_{k}-u_{k}%
^{M}\right)   dx-\int_{\partial (\theta_k^{-1}(B_{\tilde{\rho}^{M}_{k}}(0))\bigcap\Omega)} \left(  u_{k}-u_{k}%
^{M}\right) \frac{\partial}{\partial\nu}\Delta u_k   d\sigma\\
&+\int_{\partial (\theta_k^{-1}(B_{\tilde{\rho}^{M}_{k}}(0))\bigcap\Omega)} \frac{\partial}{\partial\nu}   \left(  u_{k}-u_{k}%
^{M}\right) \Delta u_k   d\sigma\\
&  = \int_{\theta_k^{-1}(B_{\tilde{\rho}^{M}_{k}}(0))\bigcap\Omega}  \gamma_k \left(  u_{k}-u_{k}%
^{M}\right)   dx-\int_{ \theta_k^{-1}(B_{\tilde{\rho}^{M}_{k}}(0))\bigcap \partial\Omega} \left(  u_{k}-u_{k}%
^{M}\right) \frac{u_k\exp(\alpha_k u^2_k)}{\lambda_k}   d\sigma\\
&  \geq -\int_{\theta_k^{-1}(B_{\tilde{\rho}^{M}_{k}}(0))\bigcap \partial\Omega}\lambda_{k}^{-1}u_{k}\exp\left\{  \alpha_{k}u_{k}%
^{2}\right\}  \left(  u_{k}-u_{k}^{M}\right)  d\sigma+o_{k}\left(  1\right)\\
& \geq-\int_{B_{Rr_{k}(x_k)}\bigcap\partial \Omega}\lambda_{k}^{-1}c_{k}\exp\left\{  \alpha_{k}u_{k}%
^{2}\right\}  \left(  c_{k}-\frac{c_{k}}{M}\right)  d\sigma+o_{k}\left(  1\right)
\\
&  =\int_{B^{+}_{R}(0)\bigcap \partial\mathbb{R}_{+}^{4}}\left(  1-\frac{1}{M}\right)  \exp\left\{  \frac{u_k(x_k+r_kx)+c_{k}}%
{c_{k}}\alpha_{k}\psi_{k}\left(  x\right)  \right\}  d\sigma+o_{k}\left(  1\right)
\\
&  \geq\left(  1-\frac{1}{M}\right)  \int_{B^{+}_{R}(0)\bigcap \partial\mathbb{R}_{+}^{4}}\exp\left\{  24\pi^{2}%
\psi\left(  x\right)  \right\}  d\sigma+o_{k}\left(  1\right)  ,
\end{align*}
letting $R\rightarrow\infty$, we get
\begin{equation}
\int_{\theta_k^{-1}(B_{\tilde{\rho}^{M}_{k}}(0))\cap\Omega}  \Delta u_{k}\Delta\left(  u_{k}-u_{k}%
^{M}\right)   dx\geq1-\frac
{1}{M}+o_{k}\left(  1\right)  . \label{inside}%
\end{equation}
Observing that%
\begin{align*}
&  \int_{%
\Omega} \left\vert \Delta u_{k}^{M}\right\vert ^{2}dx\\
&  =\int_{\theta_k^{-1}(B_{\tilde{\rho}^{M}_{k}}(0))\bigcap\Omega}\left\vert \Delta v_{k}^{M}\right\vert ^{2}%
dx+\int_{%
\Omega\backslash \theta_k^{-1}(B_{\tilde{\rho}^{M}_{k}}(0))}\left\vert \Delta u_{k}\right\vert ^{2}%
dx\\
&  =\int_{\theta_k^{-1}(B_{\tilde{\rho}^{M}_{k}}(0))\bigcap\Omega}\left\vert \Delta v_{k}^{M}\right\vert ^{2}%
dx+1-\int
_{\theta_k^{-1}(B_{\tilde{\rho}^{M}_{k}}(0))\bigcap\Omega}\left\vert \Delta u_{k}\right\vert ^{2}dx\\
&  =\int_{\theta_k^{-1}(B_{\tilde{\rho}^{M}_{k}}(0))\bigcap\Omega}\left\vert \Delta v_{k}^{M}\right\vert ^{2}%
dx+1-\int
_{\theta_k^{-1}(B_{\tilde{\rho}^{M}_{k}}(0))\bigcap\Omega}\Delta u_{k}\Delta\left(  u_{k}-u_{k}^{M}\right)  dx\\
& -\int
_{\theta_k^{-1}(B_{\tilde{\rho}^{M}_{k}}(0))\bigcap\Omega}\Delta u_{k}\Delta u_{k}^{M}dx,
\end{align*}
by (\ref{inside}) and (\ref{test}), we have
\begin{align*}
&  \int_{%
\Omega}  \left\vert \Delta u_{k}^{M}\right\vert ^{2}  dx\\
&  \leq\frac{1}{M}+\frac{1}{2}\int_{\theta_k^{-1}(B_{\tilde{\rho}^{M}_{k}}(0))\cap\Omega}\left\vert \Delta v_{k}%
^{M}\right\vert ^{2}dx-\frac{1}{2}\int_{\theta_k^{-1}(B_{\tilde{\rho}^{M}_{k}}(0))\cap\Omega}\Delta u_{k}\Delta v_{k}^{M}dx+o_k(1)\\
&  \leq\frac{1}{M}+\frac{1}{2}\int_{\theta_k^{-1}(B_{\tilde{\rho}^{M}_{k}}(0))\cap\Omega}\Delta v_{k}^{M}\Delta\left(
v_{k}^{M}-u_{k}\right)  dx+o_k(1)\\
&   =\frac{1}{M} +o_k(1).
\end{align*}
\end{proof}

\begin{lemma}
\label{concentre}We have
\begin{align*}
\underset{k\rightarrow\infty}{\lim}\int_{%
\partial\Omega}  \exp\left(  \alpha_{k}u_{k}^{2}\right)  d\sigma  & =\underset{L\rightarrow\infty}{\lim}\underset
{k\rightarrow\infty}{\lim}\int_{B_{Lr_{k}(x_k)}\cap\partial\Omega}  \exp\left(  \alpha_{k}%
u_{k}^{2}\right)  d\sigma\\
& =\underset{k\rightarrow\infty}{\lim}\frac{-\lambda_{k}}{c_{k}^{2}}+|\partial\Omega|%
\end{align*}
and consequently,
\[
\frac{-\lambda_{k}}{c_{k}}\rightarrow\infty\text{ and }\underset{k}{\sup
}\frac{-c_{k}^{2}}{\lambda_{k}}<\infty.
\]

\end{lemma}

\begin{proof}
By Lemma \ref{truncation} and Lemma \ref{concen}, we have
\begin{align*}
\int_{%
\partial\Omega}  \exp\left(  \alpha_{k}u_{k}^{2}\right)  d\sigma& = \int_{\partial\Omega \bigcap \theta_k^{-1}(B_{\tilde{\rho}^{M}_{k}}(0))} \exp\left(  \alpha_{k}u_{k}^{2}\right)
  d\sigma+\int_{%
\partial\Omega\setminus \theta_k^{-1}(B_{\tilde{\rho}^{M}_{k}}(0))}   \exp\left(  \alpha_{k}\left(  u_{k}^{M}\right)  ^{2}\right) d\sigma\\
&  \leq \frac{-M^2\lambda_k(1+o_k(1))}{c_k^2}\int_{\theta_k^{-1}(B_{\tilde{\rho}^{M}_{k}}(0))\bigcap\partial\Omega}\frac{u_k^2}{-\lambda_k} \exp\left(  \alpha_{k}u_{k}^{2}\right)
  d\sigma+|\partial\Omega|\\
 &\leq-(1+o_k(1))M^2\frac{\lambda_k}{c_k^2}+|\partial\Omega|,
\end{align*}
Let $k\rightarrow+\infty$ and $M\rightarrow 1$, we derive that $\lim\limits_{k\rightarrow+\infty}\int_{%
\partial\Omega}  \exp\left(  \alpha_{k}u_{k}^{2}\right)  d\sigma\leq-\underset{k\rightarrow\infty}{\lim}\frac{\lambda_{k}%
}{c_{k}^{2}}+|\partial\Omega|$??

On the other hand, we also have
\begin{align*}
\int_{\partial\Omega} \exp\left( \alpha_{k}u_{k}^{2}\right)   d\sigma &=\left(\int_{\partial\Omega \backslash B_{Rr_k}(x_k)}+\int_{ B_{Rr_k}(x_k)\bigcap\partial\Omega} \right)\exp\left( \alpha_{k}u_{k}^{2}\right)   d\sigma\\
& \geq  |\partial\Omega|-|B_{Rr_k}\bigcap\partial\Omega|-\frac{\lambda_{k}}{c_{k}^{2}} \int_{ B_{R}(0)\bigcap \partial \mathbb{R}_{+}^4} \exp\left(  \psi_k+o_k(1)\right)   d\sigma.\\
\end{align*}
Letting $k\rightarrow+\infty$ and $R\rightarrow+\infty$, we derive that
$\lim\limits_{k\rightarrow+\infty}\int_{%
\partial\Omega}  \exp\left(  \alpha_{k}u_{k}^{2}\right) d\sigma\geq-\underset{k\rightarrow\infty}{\lim}\frac{\lambda_{k}%
}{c_{k}^{2}}+|\partial\Omega|$. Combining the above estimates, we accomplish the proof of Lemma \ref{concentre}.
\end{proof}

\begin{lemma}\label{delta}
For any $\varphi\in C^{\infty}\left(\partial
\Omega\right)  $, one has
\begin{equation}
-\underset{k\rightarrow\infty}{\lim}\int_{\partial\Omega}\varphi\left(  x\right)  \frac{c_{k}u_{k}}{\lambda_{k}}
\exp\left( \alpha_{k}u_{k}^{2}\right)
d\sigma=\varphi\left(  p\right)  .\ \label{concentra}%
\end{equation}

\end{lemma}

\begin{proof} For any fixed $M>1$, and $k$ large enough, we divide $\partial\Omega$ into the following three parts:
 $$\Omega _{1}=\left(\theta_k^{-1}(B_{\tilde{\rho}^{M}_{k}}(0)) \setminus  B_{R r_{k}}\left(x_{k}\right)\right)\bigcap\partial\Omega,  \Omega_{2}=\partial\Omega \setminus  \theta_k^{-1}(B_{\tilde{\rho}^{M}_{k}}(0)),  \Omega _{3}=B_{R r_{k}}\left(x_{k}\right)\bigcap\partial\Omega,$$
and split the integral as follows:
\begin{align}\nonumber
\int_{\partial\Omega}\varphi\left(  x\right)  \frac{c_{k}u_{k}}{-\lambda_{k}}
\exp\left(  \alpha_{k}u_{k}^{2}\right)   d\sigma
&  =\left(\int_{\Omega_1}+\int_{\Omega_2}+\int_{\Omega_2}\right)\varphi\left(  x\right)  \frac{c_{k}u_{k}}{-\lambda_{k}}
\exp\left(  \alpha_{k}u_{k}^{2}\right)   d\sigma\\
& =I_1+I_2+I_3.
\end{align}
For $I_1$, we have
\begin{align}\label{I1}\nonumber
|I_1|&\leq M\sup _{\partial \Omega}|\varphi|\int_{\Omega_1}  \frac{ u^2_{k}}{-\lambda_{k}}
\exp\left(  \alpha_{k}u_{k}^{2}\right)   d\sigma\\ \nonumber
&\leq  M\sup _{\partial \Omega}|\varphi|(1+o_k(1)) \left(1-\int_{ B_{R r_{k}}\bigcap\partial\Omega}  \frac{ u^2_{k}}{-\lambda_{k}}
\exp\left(  \alpha_{k}u_{k}^{2}\right)d\sigma\right)\\ \nonumber
&\leq  M\sup _{\partial \Omega}|\varphi| \left(1-\int_{B^{+}_{R}\bigcap \partial\mathbb{R}^4_{+}}
\exp\left(  24\pi^2 \psi\right)d\sigma+o_k(1)\right)\\
&\rightarrow 0\ \text{ as } k,R\rightarrow\infty.
\end{align}
Next, by Lemma \ref{cufoff mod}, H{\"o}lder's inequality, Sobolev imbedding theorem and Lemma \ref{concentre}, we have
\begin{align}\label{I2}\nonumber
\left|I_{2}\right| & \leq \sup _{\partial \Omega}|\varphi| \frac{c_{k}}{-\lambda_{k}} \int_{\partial \Omega}\left|u_{k}\right| e^{\alpha_k\left(u_{k}^{M}\right)^{2}}d\sigma \\ \nonumber
& \leq  \sup _{\partial \Omega}|\varphi| \frac{c_{k}}{-\lambda_{k}}\left\|u_{k}\right\|_{L^{p'}(\partial \Omega)}\left\|e^{\alpha_k\left(u_{k}^{M}\right)^{2}}\right\|_{L^{p}(\partial \Omega)} \\
& \leq c \sup _{\partial \Omega}|\varphi| |\frac{c_{k}}{\lambda_{k}}|\rightarrow 0 \text{ as } k\rightarrow\infty,
\end{align}
for some $p>1$ and $p'$ with $\frac{1}{p}+\frac{1}{p'}=1$.

Finally, we have

 \begin{align}\label{I3}\nonumber
 I_{3}&=\int_{ B_{Rr_{k}}\bigcap\partial\Omega}\varphi\left(  x\right)  \frac{c_{k}u_{k}}{-\lambda_{k}}
\exp\left(  \alpha_{k}u_{k}^{2}\right)   d\sigma\\ \nonumber
&=\int_{ B^{+}_{R}\cap \partial\mathbb{R}^4_+}\varphi\left(  r_kx+x_k\right)  \exp\left\{
 (\phi_k+1)\alpha_{k} \psi_{k}\left(  x\right)  \right\}  dx+o_{k}\left(  1\right)\\ \nonumber
&= \varphi\left( p\right) \int_{B^{+}_{R}\cap \partial\mathbb{R}^4_+}\exp\left\{  24\pi^{2}%
\psi\left(  x\right)  \right\}  d\sigma+o_{k}\left(  1\right)\\
&= \varphi\left( p\right) +o_{k,R}(1).
\end{align}
Combining (\ref{I1}), (\ref{I2}) and (\ref{I3}), we obtain (\ref{concentra}) and the proof is finished.

\end{proof}

\begin{lemma}\label{gre}
$c_ku_k\rightharpoonup G$ weakly in $W^{2,q}(\Omega)$ for any $1<q<2$, furthermore, for any $\Omega'\Subset\overline{\Omega}\setminus{p}$, we have $c_ku_k\rightarrow G$ in $C^{\infty}(\overline{\Omega'})$,  where $G$ satisfies \begin{equation}\label{green}
\begin{cases}& \Delta^2 G=\delta_{p}-\frac{1}{|\Omega|} \text{ in } \Omega,\\&
\int_{\Omega}G=0, \frac{\partial G}{\partial\nu}=0, \frac{\partial\Delta G}{\partial\nu}|_{\partial\Omega\setminus\{p\}}=0. \end{cases}
\end{equation}
Moreover, we have \begin{equation}
G=-\frac{1}{4\pi^{2}}\ln\left\vert x-p\right\vert +A_p+\varphi\left(  x\right)
,\label{Green1}%
\end{equation}
where $A_p$ is some constant depending on $p$, $\varphi\left( x\right)  \in
C^{3}\left(
\Omega\right)\bigcap C^1(\overline{\Omega})$ and $\varphi\left(  p\right)  =0$.

\end{lemma}

\begin{proof} From \eqref{fen}, we have
\begin{equation}\label{veq}
\begin{cases}
 \Delta^2 (c_ku_k)=c_k\gamma_k, & \forall x\in {\mathop \Omega
}, \\
    \frac{\partial}{\partial \nu}\Delta (c_k u_k)=\frac{c_ku_k\exp(\alpha_k u^2_k)}{\lambda_k}, & \forall x\in \partial\Omega.
\end{cases}
\end{equation}

Integrating both sides on $\Omega$, one has

\begin{align}\nonumber
\int_{\Omega}c_k\gamma_kdx &=\int_{\Omega}\Delta^2 (c_ku_k)dx\\\nonumber
&=\int_{\partial\Omega}\frac{\partial\Delta (c_ku_k)}{\partial\nu} d\sigma\\
&=\int_{\partial\Omega}\frac{c_ku_k\exp(\alpha_k u^2_k)}{\lambda_k}d\sigma,
\end{align}
which together with Lemma \ref{delta} gives $c_k\gamma_k\rightarrow-\frac{1}{|\Omega|}$ as $k\rightarrow\infty$.
For any $q\in (1,2)$, we have
\begin{equation*}
\int_{\Omega}|\Delta c_ku_k|^qdx =\sup\{\int_{\Omega}\Delta (c_ku_k)\Delta \varphi dx: \parallel \varphi\parallel_{W^{2,q'}}=1\},
\end{equation*}
where $\frac{1}{q}+\frac{1}{q'}=1$. By the Sobolev embedding theorem, we have $\sup_{x\in\Omega}|\varphi(x)|<\infty$.  Using Lemma \ref{delta}, we have
\begin{align}\label{div1}
\int_{\Omega}\Delta(c_ku_k)\Delta\varphi dx &=\int_{\Omega}\Delta^2 (c_ku_k)\varphi dx-\int_{\partial\Omega}\frac{\partial\Delta (c_ku_k)}{\partial\nu}\varphi d\sigma\\\nonumber
&=\int_{\Omega}c_k\gamma_k\varphi(x) dx-\int_{\partial\Omega}\frac{c_ku_k\varphi\exp(\alpha_k u^2_k)}{\lambda_k}d\sigma\\\nonumber
&=-\frac{1}{|\Omega|}\int_{\Omega}\varphi(x)dx+\varphi(p)+o_k(1)\\\nonumber
&\leq c\sup_{x\in\Omega}|\varphi(x)|<c,
\end{align}
which implies that
$$\int_{\Omega}|\Delta c_ku_k|^qdx<c.$$
Combining this and the condition $\int_{\Omega}c_ku_kdx=0$, $\int_{\Omega}u_kdx=0$, we derive that $c_ku_k$ is bounded in $W^{2,q}(\Omega)$ for any $1\leq q<2$. Thus, there exists some $G\in W^{2,q}(\Omega)$ such that $c_ku_k\rightharpoonup G$ in $W^{2,q}(\Omega)$ as $k\rightarrow\infty$. Now, letting $k\rightarrow\infty$ in (\ref{div1}), we have
\begin{align}\nonumber
\int_{\Omega}\Delta G\Delta\varphi dx =- \frac{1}{|\Omega|}\int_{\Omega}\varphi(x)dx+\varphi(p).
\end{align}
Combining the assumptions on $u_k$,  (\ref{green}) is proved.
\vskip0.1cm

For any $\Omega'\Subset \overline{\Omega}\setminus{p}$, we can choose some function $\phi\in C^{\infty}(\mathbb{R}^4)$ such that $\phi(x)=1 $ for $x\in\Omega' $ and $\phi(x)=0 $ for $x$ belonging to a small neighborhood of $p$. By Lemma \ref{concen}, we know that $\phi u_{k} \rightarrow 0$ in $L^{2}(\Omega')$ as $k \rightarrow +\infty$. This together with the convergence $\Delta u_{k} \rightarrow 0$ in $L^{2}(\Omega')$ as $k \rightarrow \infty$ implies that $e^{\alpha_k u_{k}^{2}}$ is uniformly bounded in $L^{s}( \overline{\Omega'})$ for any $s>1$. Standard elliptic regularity gives that $c_{k} u_{k} \rightarrow G$ in $C^{k}( \overline{\Omega'})$ for any positive integer $k$.
\vskip0.1cm

Next, we prove (\ref{Green1}). Fix $r>0$, without loss of generality, we assume $p=0$, and choose some cutoff function $\phi\in C_{0}^{\infty}\left(
B_{2r}\left(  0\right)  \right)  $ such that $\phi=1$ in $B_{r}\left(
0\right)  $. Let%
\[
g\left(  x\right)  =G\left(  x\right)  +\frac{1}{4\pi^{2}}\phi\left(
x\right)  \ln\left\vert x\right\vert .
\]
Then we have
\[
\Delta^{2}g\left(  x\right)  =f\text{ in }%
\Omega\text{,}%
\]
where
\begin{align*}
f\left(  x\right)   &  =\frac{1}{4\pi^{2}}\left(  \Delta^{2}\phi\cdot
\ln\left\vert x\right\vert +2\nabla\Delta\phi\cdot\nabla\ln\left\vert
x\right\vert +\right.  \\
&  \left.  +2\Delta\left(  \nabla\phi\cdot\nabla\ln\left\vert x\right\vert
\right)  +2\nabla\phi\cdot\nabla\Delta\ln\left\vert x\right\vert +\phi
\cdot\Delta^{2}\ln\left\vert x\right\vert \right)  +\delta\left(
x\right)  -\frac{1}{|\Omega|}.
\end{align*}
Since $\frac{1}{4\pi^{2}}\phi\cdot\Delta^{2}\ln\left\vert x\right\vert
=\delta\left(  x\right)  $ in $%
\mathbb{R}
^{4}_{+}$, a careful computation yields
\begin{align*}
f\left(  x\right)   &  =\frac{1}{4\pi^{2}}\left(  \Delta^{2}\phi\cdot
\ln\left\vert x\right\vert +2\nabla\Delta\phi\cdot\nabla\ln\left\vert
x\right\vert +\right.  \\
&  \left.  +2\Delta\left(  \nabla\phi\cdot\nabla\ln\left\vert x\right\vert
\right)  +2\nabla\phi\cdot\nabla\Delta\ln\left\vert x\right\vert \right)  -\frac{1}{|\Omega|}.
\end{align*}
Observing $G\in W^{2,s}\left(
\Omega\right)  $ for any $1<s<2$, we obtain $f\left(  x\right)  \in L%
^{p}\left(
\Omega\right)  $ for any $p>2$. By \ the standard regularity theory, we get
$g\left(  x\right)  \in C_{loc}^{3}(\Omega)\bigcap C^1(\overline{\Omega})$. Let $A_p=g\left(  0\right)  $ and%
\[
\varphi\left(  x\right)  =g\left(  x\right)  -g\left(  0\right)  +\frac
{1}{4\pi^{2}}\left(  1-\phi\right)  \ln\left\vert x\right\vert .
\]
Then\ we have
\begin{equation}
G=-\frac{1}{4\pi^{2}}\ln\left\vert x\right\vert +A_p+\varphi\left(  x\right),%
\end{equation}
where $A_p$ is some constant depending on $p$, $\varphi\left( x\right)  \in
C^{3}\left(
\Omega\right)\bigcap C^1(\overline{\Omega})$ and $\varphi\left(  0\right)  =0$, and the proof is finished.

\end{proof}

\subsection{Neck analysis}\label{Neck}

In this subsection, we will use the capacity technique to derive the upper bound of $I_{12\pi^2}(u_k)$ when $c_k\rightarrow\infty$. The capacity technique applied to the existence of extremals for Adams inequalities was first used by Lu and Yang in \cite{lu-yang 1}, and was improved by DelaTorre and Mancini in \cite{DelaTorre} by comparing the Dirichlet energy of maximizing sequence with the energy of a suitable poly-harmonic function.

Based on Lemma \ref{concentre}, we only need to give the sharp upper bound of  $\underset{k\rightarrow\infty}{\lim}\frac{-\lambda_k}{c^2_k}$. Let us fix a large $R > 0$ and a small $\delta > 0$ and consider the annular region $$A_{k}(R, \delta):=\left\{x \in \Omega: r_k R \leq\left|x-x_k\right| \leq \delta\right\}$$
Our strategy is to compare the Dirichlet energy of $u_k$ on $A_{k}(R, \delta)$ with  the energy of the following
function $$\mathcal{W}_k(x):=-\frac{1}{4\pi^2 c_k} \left(\log\left|x-x_k \right|+\rho_k(x)\right),$$ where $\rho_k(x)\in C^\infty(\bar{\Omega})$ is chosen such that $$\frac{\partial \mathcal{W}_k(x)}{\partial\nu}=\frac{\partial\Delta\mathcal{W}_k(x)}{\partial\nu}=0, \text{for } x\in \partial\Omega$$ and $\|\rho_k(x)\|_{C^3}=O(\delta)$.

As a consequence of Lemma \ref{limiteq}, on $\partial B_{R r_k}\left(x_k\right)\bigcap\Omega$, we have that
\begin{align*}
u_k(x)&=c_k+\frac{\psi\left(\frac{x-x_k}{r_k}\right)}{c_k}+o\left(c_{k}^{-1}\right)\\
&=c_k-\frac{1}{4\pi^2c_k}\log R -\frac{1}{6\pi^2c_k}\log{\frac{\pi}{2}}+\frac{O\left(R^{-1}\right)}{c_k}+o\left(c_{k}^{-1}\right),
\end{align*}
 provided $k$ large enough. Similarly, a direct computation also gives
$$\begin{aligned} \Delta^{\frac{j}{2}}u_k &=\frac{\left(\Delta^{\frac{j}{2}} \psi\right)\left(\frac{x-x_k}{r_{k}}\right)}{r_{k}^{j} c_{k}}+o\left(r_{k}^{-j} c_{k}^{-1}\right) \\ &=-\frac{ K_{2, \frac{j}{2}}}{4\pi^2 r_{k}^{j} c_{k} R^{j}} e_{ j}\left(x-x_k\right)+\frac{O\left(R^{-j-1}\right)}{r_{k}^{j} c_{k}}+o\left(r_{k}^{-j} c_{k}^{-1}\right), \end{aligned}$$
for any $1\leq j\leq 3$, where \[{K_{2,\frac{j}{2}}}
 =
\begin{cases}
1,  & \mbox{if }j=1, \\
2, & \mbox{if }j=2,\\
-4,& \mbox{if }j=3\\
\end{cases}\]
and $e_{j}(x):=\left\{\begin{array}{l}1, j \text { even, } \\ \frac{x}{|x|}, j \text { odd. }\end{array}\right.$

Recalling the definition of $ \mathcal{W}_k$, we have on $\partial B_{R r_k}\left(x_k\right)\bigcap\Omega$ that
\begin{equation}\label{add6}\mathcal{W}_k=\frac{\alpha_k}{12\pi^2}c_k-\frac{1}{12\pi^2 c_k}\log\frac{-\lambda_k}{c^2_k}-\frac{1}{4\pi^2 c_k}\log R+\frac{O(\delta)}{c_k}\end{equation}
and
 \begin{equation}\label{add7}
\Delta^{\frac{j}{2}} \mathcal{W}_{k}=-\frac{ K_{2, \frac{j}{2}}}{4\pi^2 c_{k} r_{k}^{j} R^{j}} e_{ j}(x-x_k)+\frac{O(\delta)}{c_k}, \quad \text { for any } 1 \leq j \leq 3.
\end{equation}
Hence, we conclude that on $\partial B_{R r_k}\left(x_k\right)\bigcap\Omega$,
$$
u_k(x)-\mathcal{W}_{k}=\frac{1}{12\pi^2 c_k}\log\frac{-\lambda_k}{c^2_k}-\frac{1}{6\pi^2c_k}\log{\frac{\pi}{2}}+\frac{O\left(R^{-1}\right)}{c_k}+\frac{O(\delta)}{c_{k}}+o\left(c_{k}^{-1}\right)+(1-\frac{\alpha_k}{12\pi^2})c_k
$$
and $$
\Delta^{\frac{j}{2}}\left(u_k-\mathcal{W}_{k}\right)=\frac{O\left(R^{-j-1}\right)}{r_{k}^{j} c_{k}}+o\left(r_{k}^{-j} c_{k}^{-1}\right), \quad \text { for any } 1 \leq j \leq 3.
$$
Similarly, in view of Lemma \ref{gre}, we also derive that on $\partial B_{\delta}\left(x_k\right)\bigcap \Omega$,
$$u_k(x)-\mathcal{W}_{k}=\frac{A_p}{c_{k}}+\frac{O(\delta)}{c_{k}}+o\left(c_{k}^{-1}\right)
$$
and
$$
\Delta^{\frac{j}{2}}\left(u_k(x)-\mathcal{W}_{k}\right)=\frac{O(1)}{c_{k}}+o\left(c_{k}^{-1}\right), \quad \text { for any } 1 \leq j \leq 3,
$$
where we have also used that $
\frac{\left|x-x_{k}\right|}{\left|x-p\right|} \rightarrow 1
$ uniformly on $\partial B_{\delta}\left(x_k\right)$.

Now, we  compare $\left\|\Delta u_{k}\right\|_{L^{2}\left(A_{k}(R, \delta)\right)}$ and $\left\|\Delta \mathcal{W}_{k}\right\|_{L^{2}\left(A_{k}(R, \delta)\right)}$.
Obviously,
\begin{equation}\label{A}
\left\|\Delta u_{k}\right\|_{L^{2}\left(A_{k}(R, \delta)\right)}^{2}-\left\|\Delta \mathcal{W}_{k}\right\|_{L^{2}\left(A_{k}(R, \delta)\right)}^{2} \geq 2 \int_{A_{k}(R, \delta)} \Delta\left(u_{k}-\mathcal{W}_{k}\right) \cdot \Delta \mathcal{W}_{k} d x.
\end{equation}

\textbf{Step 1}. Estimates for the RHS of \eqref{A}.
\bigskip

Integrating by parts,  the integral in the right-hand side equals to

 \begin{align*}
   2\int_{A_{k}(R, \delta)} \Delta\left(u_{k}-\mathcal{W}_{k}\right) \cdot \Delta \mathcal{W}_{k} d x&=-2\int_{\partial A_{k}(R, \delta)\backslash\partial \Omega} \nu \cdot\left(\left(u_{k}-\mathcal{W}_{k}\right) \Delta^{\frac{3}{2}} \mathcal{W}_{k}\right) d \sigma\\
   &+2\int_{\partial A_{k}(R, \delta)\backslash\partial \Omega} \nu \cdot\left(\Delta^{\frac{1}{2}}\left(u_{k}-\mathcal{W}_{k}\right) \Delta \mathcal{W}_{k}\right) d \sigma,
 \end{align*}
where we have used the fact $\frac{\partial\mathcal{W}_{k}}{\partial\nu}=\frac{\partial\Delta\mathcal{W}_{k}}{\partial\nu}=0$ on $\partial A_{k}(R, \delta)\bigcap\partial \Omega$.

On  $\partial B_{R r_k}\left(x_k\right)\bigcap \Omega$, we have

$$
\begin{aligned}
\left(u_{k}\right.&\left.-\mathcal{W}_{k}\right) \Delta^{\frac{3}{2}} \mathcal{W}_{k} \cdot \nu \\
=&\frac{1}{4\pi^2}\left(\frac{1}{12\pi^2 c_{k}^{2}} \log \left(\frac{-\lambda_k}{ c_{k}^{2}}\right)-\frac{1}{6\pi^2c_k^2}\log{\frac{\pi}{2}}+\frac{O\left(R^{-1}\right)}{c_{k}^{2}}+(1-\frac{\alpha_k}{12\pi^2})+\frac{O(\delta)}{c_{k}^{2}}+o\left(c_{k}^{-2}\right)\right) \frac{K_{2, \frac{3}{2}}}{\left(r_{k} R\right)^{3}} \\
=&-\frac{1}{\pi^2\left(r_{k} R\right)^{3}}\left(\frac{1}{12\pi^2 c_{k}^{2}} \log \left(\frac{-\lambda_k}{ c_{k}^{2}}\right)-\frac{1}{6\pi^2c_k^2}\log{\frac{\pi}{2}}+     \frac{O\left(R^{-1}\right)}{c_{k}^{2}}+(1-\frac{\alpha_k}{12\pi^2})+\frac{O(\delta)}{c_{k}^{2}}+o\left(c_{k}^{-2}\right)\right)
\end{aligned}
$$
and $$
\Delta^{\frac{1}{2}}\left(u_{k}-\mathcal{W}_{k}\right) \Delta \mathcal{W}_{k} \cdot \nu=\left(\frac{O\left(R^{-1}\right)}{c_{k}^{2}}+o\left(c_{k}^{-2}\right)\right) O\left(r_{k} R\right)^{-3}.
$$

Similarly, on  $\partial B_{\delta}\left(x_k\right)\bigcap\Omega$, we have
$$
\left(u_{k}-\mathcal{W}_{k}\right) \Delta^{\frac{3}{2}} \mathcal{W}_{k} \cdot \nu=\frac{1}{\pi^2 \delta^{3}}\left(\frac{A_p}{c_{k}^{2}}+\frac{O(\delta)}{c_{k}^{2}}+o\left(c_{k}^{-2}\right)\right)
$$
and
$$
\Delta^{\frac{1}{2}}\left(u_{k}-\mathcal{W}_{k}\right) \Delta \mathcal{W}_{k} \cdot \nu=\left(\frac{O(1)}{c_{k}^{2}}+o\left(c_{k}^{-2}\right)\right) O\left(\delta^{-2}\right).
$$
Then we can obtain
\begin{align*}
 & \int_{A_{k}(R, \delta)} \Delta\left(u_{k}-\mathcal{W}_{k}\right) \cdot \Delta \mathcal{W}_{k} d x  \\
   &= \frac{1}{12\pi^2 c_{k}^{2}} \log \left(\frac{-\lambda_k}{ c_{k}^{2}}\right)-\frac{1}{6\pi^2c_k^2}\log{\frac{\pi}{2}}-\frac{A_p}{c^2_k}+\frac{O\left(R^{-1}\right)}{c_{k}^{2}}+\frac{O(\delta)}{c_{k}^{2}}+(1-\frac{\alpha_k}{12\pi^2})+o\left(c_{k}^{-2}\right).
\end{align*}
Combining the above estimates, we derive that
\begin{align}\label{esti1}
& \left\|\Delta u_{k}\right\|_{L^{2}\left(A_{k}(R, \delta)\right)}^{2}-\left\|\Delta \mathcal{W}_{k}\right\|_{L^{2}\left(A_{k}(R, \delta)\right)}^{2} \nonumber\\ & \geq  \frac{1}{6\pi^2 c_{k}^{2}} \log \left(\frac{-\lambda_k}{ c_{k}^{2}}\right)-\frac{1}{3\pi^2c_k^2}\log{\frac{\pi}{2}}-\frac{2A_p}{c^2_k}+\frac{O\left(R^{-1}\right)}{c_{k}^{2}}+\frac{O(\delta)}{c_{k}^{2}}+(2-\frac{\alpha_k}{6\pi^2})+o\left(c_{k}^{-2}\right).
\end{align}

\textbf{Step 2}. Estimates for $\left\|\Delta u_{k}\right\|_{L^{2}\left(A_{k}(R, \delta)\right)}^{2}$.

We rewrite $\left\|\Delta u_{k}\right\|_{L^{2}\left(A_{k}(R, \delta)\right)}^{2}$ as  follows:
\begin{equation}\label{add5}
\left\|\Delta u_{k}\right\|_{L^{2}\left(A_{k}(R, \delta)\right)}^{2}=1-\int_{\Omega \setminus B_{\delta}\left(x_k\right)}\left|\Delta u_{k}\right|^{2} d x-\int_{\Omega\cap B_{Rr_{k}}\left(x_k\right)}\left|\Delta u_{k}\right|^{2} d x.
\end{equation}
Since $$\Delta^{\frac{1}{2}}(\log |x|)=\frac{x}{|x|^{2}}, \Delta(\log |x|)=\frac{2}{|x|^{2}}, \Delta^{1+\frac{1}{2}}(\log |x|)=-4 \frac{x}{|x|^{4}},$$ we have

 \begin{align}\label{G1} \nu \cdot G(\delta) \Delta^{3 / 2} G(\delta) &=-\left(-\frac{1}{4 \pi^{2}} \ln |\delta|+A_p+o_{\delta}(1)\right)\left(\frac{1}{ \pi^{2}} \cdot\frac{1}{\delta^{3}}+O(1)\right) \nonumber \\ &=-\frac{1}{ \pi^{2}} \frac{1}{\delta^{3}}\left(-\frac{1}{4\pi^{2}} \ln \delta+A_p+o_{\delta}(1)\right) \end{align}
 and
 \begin{align}\label{G2} \nu \cdot \Delta^{1 / 2} G(\delta) \Delta G(\delta) &=-\left(-\frac{1}{4\pi^{2}} \frac{1}{\delta}+O(1)\right)\left(-\frac{1}{4 \pi^{2}} \frac{2}{\delta^{2}}+O(1)\right) \nonumber\\ &=-\frac{1}{8 \pi^{4}} \frac{1}{\delta^{3}}(1+o_{\delta}(1)). \end{align}

Since $$
 \begin{aligned}
\int_{\Omega \setminus B_{\delta}(x_k)}\left|\Delta G\right|^{2} d x=\int_{\Omega \cap \partial B_{\delta}(x_k)}\nu\left(-G \Delta^{3 / 2} G+\Delta^{1 / 2} G \Delta G\right) d \sigma,
\end{aligned}
$$
 we have by Lemma \ref{gre},
\begin{equation}\label{add4}\int_{\Omega \backslash  B_{\delta}(x_k)}\left|\Delta u_{k}\right|^{2} d x=\frac{1}{c_k^2}\left(-\frac{1}{4\pi^2}\log\delta-\frac{1}{8\pi^2}+A_p+o_{\delta}(1)+o_k(1)\right).\end{equation}

By Lemma \ref{limiteq}, we derive that
\begin{eqnarray}\label{add2}
\int_{\Omega \cap B_{Rr_{k}(x_k)}}\left\vert \Delta u_{k}\right\vert ^{2}dx &=&\frac{1}{%
c_{k}^{2}}\int_{B_{R}^{+}}\left\vert \Delta \psi \right\vert ^{2}dx+o\left(
\frac{1}{c_{k}^{2}}\right) \nonumber\\
&=&\frac{1}{c_{k}^{2}}\left( \int_{\partial B_{R}^{+}}\nu \left( \Delta ^{%
\frac{1}{2}}\psi \Delta \psi -\psi \Delta ^{\frac{3}{2}}\psi \right) d\sigma
\right) +o\left( \frac{1}{c_{k}^{2}}\right) \nonumber \\
&:=&\frac{1}{c_{k}^{2}}\left( I-II\right) +o\left( \frac{1}{c_{k}^{2}}\right).
\end{eqnarray}%

Observe that on $\partial B_{R}^{+}\cap \mathbb{R}_{+}^{4}$, we also have
\begin{equation*}
\psi \left( x\right) =-\frac{1}{6\pi ^{2}}\log \left( \frac{\pi }{2}\right) -%
\frac{1}{4\pi ^{2}}\log R+O\left( \frac{1}{R}\right),
\end{equation*}

\begin{equation*}
\nu \Delta ^{\frac{1}{2}}\psi \left( x\right) =-\frac{1}{4\pi ^{2}}\frac{1}{R%
}+O\left( \frac{1}{R^{2}}\right)
\end{equation*}
and

\begin{eqnarray*}
\nu \Delta ^{\frac{3}{2}}\psi &=&-\frac{1}{4\pi ^{2}}\left( \frac{-4}{\left(
\left( t+\left( \frac{2}{\pi }\right) ^{\frac{2}{3}}\right)
^{2}+|x'|^{2}\right) ^{2}}\right) \frac{\left( x',t+\left( \frac{2}{\pi }\right)
^{\frac{2}{3}}\right) \cdot \left( x',t\right) }{R}+O\left( \frac{1}{R^{4}}%
\right) \\
&=&\frac{1}{\pi ^{2}}\left( \frac{1}{\left( \left( t+\left( \frac{2}{\pi }%
\right) ^{\frac{2}{3}}\right) ^{2}+|x'|^{2}\right) ^{2}}\right) \frac{\left(
x',t+\left( \frac{2}{\pi }\right) ^{\frac{2}{3}}\right) \cdot \left(
x',t\right) }{R}+O\left( \frac{1}{R^{4}}\right) \\
&=&\frac{1}{\pi ^{2}}\left( \frac{1}{R^{4}}+O\left( \frac{1}{R^{5}}\right)
\right) \left( R+O\left( 1\right) \right) +O\left( \frac{1}{R^{4}}\right) \\
&=&\frac{1}{\pi ^{2}}\frac{1}{R^{3}}+O\left( \frac{1}{R^{4}}\right).
\end{eqnarray*}
Hence we can write
\begin{equation*}
II=\int_{\partial B_{R}^{+}\cap \mathbb{R}_{+}^{4}}\nu \psi \Delta ^{\frac{3}{2}}\psi
d\sigma +\int_{\partial B_{R}^{+}\cap \partial \mathbb{R}_{+}^{4}}\nu \psi \Delta ^{%
\frac{3}{2}}\psi d\sigma :=II_{1}+II_{2},
\end{equation*}%
where%
\begin{eqnarray*}
II_{1} &=&\int_{\partial B_{R}^{+}\cap  \mathbb{R}_{+}^{4}}\nu \psi \Delta ^{%
\frac{3}{2}}\psi d\sigma \\
&=&\pi ^{2}R^{3}\left( -\frac{1}{6\pi ^{2}}\log \left( \frac{\pi }{2}\right)
-\frac{1}{4\pi ^{2}}\log R+O\left( \frac{1}{R}\right) \right) \cdot \\
&&\left( \frac{1}{\pi ^{2}}\frac{1}{R^{3}}+O\left( \frac{1}{R^{4}}\right)
\right) \\
&=&-\frac{1}{4\pi ^{2}}\log R-\frac{1}{6\pi ^{2}}\log \left( \frac{\pi }{2}%
\right) +O\left( \frac{\log R}{R}\right).
\end{eqnarray*}%
\bigskip

Since $\frac{\partial }{\partial t }\Delta \psi =\exp \left(
24\pi ^{2}\psi \right) $ for $x=(x',0)\in \partial\mathbb{R}_{+}^{4}$, let $\psi_0(x')=\psi(x',0)$, we have

\begin{eqnarray*}
II_{2} &=&\int_{\partial B_{R}^{+}\cap \partial \mathbb{R}_{+}^{4}}\nu \psi \Delta ^{%
\frac{3}{2}}\psi d\sigma \\
&=&\int_{B_{R}^{3}}-\exp \left( 24\pi ^{2}\psi_0\left( x'\right)
\right) \psi_0\left( x'\right) dx' \\
&=&-\int_{\mathbb{R}^{3}}\frac{\left( \left( \frac{2}{\pi }\right) ^{\frac{2}{3}%
}\right) ^{3}}{\left( \frac{\pi }{2}\right) ^{2}\left( |x'|^{2}+\left( \left(
\frac{2}{\pi }\right) ^{\frac{2}{3}}\right) ^{2}\right) ^{3}}\psi_0%
\left( x'\right) dx'+O\left( \frac{1}{R}\right) \\
&=&-\psi \left( 0,\left( \frac{2}{\pi }\right) ^{\frac{2}{3}}\right)
+O\left( \frac{1}{R}\right) \\
&=&\frac{1}{4\pi ^{2}}\log  2 -\frac{1}{16\pi ^{2}}+O\left(
\frac{1}{R}\right),
\end{eqnarray*}
where $B_{R}^{3}$ denotes the 3 dimensional  balls with radius $R$.

So, we have
\begin{eqnarray}\label{add1}
II &=&-\frac{1}{4\pi ^{2}}\log R-\frac{1}{6\pi ^{2}}\log \left( \frac{\pi }{2%
}\right) + \frac{1}{4\pi ^{2}}\log2 -\frac{1}{16\pi ^{2}}+O\left(
\frac{\log R}{R}\right) \nonumber \\
&=&-\frac{1}{4\pi ^{2}}\log \frac{R}{2}-\frac{1}{6\pi ^{2}}\log \left( \frac{%
\pi }{2}\right) -\frac{1}{16\pi ^{2}}+O\left( \frac{\log R}{R}\right).
\end{eqnarray}

Now, we estimate $I$, and rewrite it as

\begin{eqnarray*}
I &=&\int_{\partial B_{R}^{+}}\nu \Delta ^{\frac{1}{2}}\psi \Delta \psi
d\sigma \\
&=&\int_{\partial B_{R}^{+}\cap \mathbb{R}_{+}^{4}}\nu \Delta ^{\frac{1}{2}}\psi
\Delta \psi d\sigma +\int_{\partial B_{R}^{+}\cap \partial \mathbb{R}_{+}^{4}}\nu
\Delta ^{\frac{1}{2}}\psi \Delta \psi d\sigma \\
&:=&I_{1}+I_{2}.
\end{eqnarray*}
Since on $\partial B_{R}^{+}\cap \mathbb{R}_{+}^{4}$, we have
\begin{equation*}
\nu \Delta ^{\frac{1}{2}}\psi \left( x\right) =-\frac{1}{4\pi ^{2}}\frac{1}{R%
}+O\left( \frac{1}{R^{2}}\right)
\end{equation*}
and $\Delta \psi =-\frac{1}{2\pi ^{2}} \frac{1}{R ^{2}} $, so we get

\begin{eqnarray*}
I_{1} &=&\int_{\partial B_{R}^{+}\cap \mathbb{R}_{+}^{4}}\nu \Delta ^{\frac{1}{2}%
}\psi \Delta \psi d\sigma \\
&=&\pi ^{2}R^{3}\left( -\frac{1}{4\pi ^{2}}\frac{1}{R}+O\left( \frac{1}{R^{2}%
}\right) \right) \left( -\frac{1}{4\pi ^{2}}\left( \frac{2}{R^{2}}\right)
\right) \\
&=&\frac{1}{8\pi ^{2}}+O\left( \frac{1}{R}\right).
\end{eqnarray*}
Using the fact that $\frac{\partial \psi }{\partial t}=0$ on $\partial \mathbb{R}_{+}^{4},$ we obtain
\begin{equation*}
I_{2}=\int_{\partial B_{R}^{+}\cap \partial \mathbb{R}_{+}^{4}}\nu \Delta ^{\frac{1}{2%
}}\psi \Delta \psi d\sigma =0.
\end{equation*}
Hence, we have
\begin{equation}\label{add3}
I=\frac{1}{8\pi ^{2}}+O\left( \frac{1}{R}\right).
\end{equation}

By \eqref{add2},\eqref{add1} and \eqref{add3}, we get

\begin{eqnarray}\label{addc1}
&&\int_{\Omega \cap B_{Rr_{k}(x_k)}}\left\vert \Delta u_{k}\right\vert ^{2}dx \nonumber\\
&=&\frac{1}{c_{k}^{2}}\left( \frac{1}{8\pi ^{2}}+O\left( \frac{1}{R}\right)
-\left( -\frac{1}{4\pi ^{2}}\log \frac{R}{2}-\frac{1}{6\pi ^{2}}\log
\frac{\pi }{2} \right. \right. \nonumber \\
&&\left. \left. -\frac{1}{16\pi ^{2}}+O\left( \frac{\log R}{R}\right)
\right) \right) +o\left( \frac{1}{c_{k}^{2}}\right)\nonumber \\
&=&\frac{1}{c_{k}^{2}}\left( \frac{3}{16\pi ^{2}}+\frac{1}{4\pi ^{2}}\log
\frac{R}{2}+\frac{1}{6\pi ^{2}}\log \frac{\pi }{2} \right) +%
\frac{1}{c_{k}^{2}}O\left( \frac{\log R}{R}\right).
\end{eqnarray}

Combining \eqref{add4} and \eqref{add5}, we derive that
\begin{eqnarray}\label{addb4}
\left\Vert \Delta u_{k}\right\Vert _{A_{k}\left( R,\delta \right) }^{2}
&=&1-\frac{1}{c_{k}^{2}}\left( -\frac{1}{4\pi ^{2}}\log \delta -\frac{1}{%
8\pi ^{2}}+A_{p}+o_{\delta ,k}\left( 1\right) \right) \nonumber\\
&&-\frac{1}{c_{k}^{2}}\left( \frac{3}{16\pi ^{2}}+\frac{1}{4\pi ^{2}}\log
\frac{R}{2}+\frac{1}{6\pi ^{2}}\log  \frac{\pi }{2} \right) +%
\frac{1}{c_{k}^{2}}O\left( \frac{\log R}{R}\right) \nonumber\\
&=&1-\frac{1}{c_{k}^{2}}\left( \frac{1}{16\pi ^{2}}+\frac{1}{4\pi ^{2}}\log
\frac{R}{2\delta }+\frac{1}{6\pi ^{2}}\log  \frac{\pi }{2}
+A_{p}+O\left( \frac{\log R}{R}\right)+o_{\delta,k}(1)\right).
\end{eqnarray}

\textbf{Step 3}. Estimates for $\left\|\Delta \mathcal{W}_{k}\right\|_{L^{2}\left(A_{k}(R, \delta)\right)}^{2}$.

Since
\begin{eqnarray}\label{addb1}
\int_{A_{k}\left( R,\delta \right) }\left\vert \Delta \mathcal{W}_{k}\right\vert
^{2}dx &=&-\int_{\partial A_{k}\left( R,\delta \right)
}\nu \left( \mathcal{W}_{k}\Delta ^{\frac{3}{2}}\mathcal{W}_{k}-\Delta ^{\frac{1}{2}}\mathcal{W}_{k}\Delta \mathcal{W}_{k}\right) d\sigma \nonumber
\\
&=&-\int_{\Omega \cap \partial B_ \delta(x_k) }\nu \left(
\mathcal{W}_{k}\Delta ^{\frac{3}{2}}\mathcal{W}_{k}-\Delta ^{\frac{1}{2}}\mathcal{W}_{k}\Delta \mathcal{W}_{k}\right) d\sigma \nonumber \\
&&+\int_{\Omega \cap \partial B_{ Rr_{k}(x_k))} }\nu \left(
\mathcal{W}_{k}\Delta ^{\frac{3}{2}}\mathcal{W}_{k}-\Delta ^{\frac{1}{2}}\mathcal{W}_{k}\Delta \mathcal{W}_{k}\right) d\sigma \nonumber \\
&:=&- III_{1}+III_{2}.
\end{eqnarray}

From \eqref{add6} and \eqref{add7}, we have
\begin{eqnarray*}
& III_{2}& \\ &=&\left( \left( \frac{\alpha_k}{12\pi^2}c_{k}-\frac{1}{12\pi ^{2}c_{k}}\log \frac{%
-\lambda _{k}}{c_{k}^{2}}-\frac{1}{4\pi ^{2}c_{k}}\log R\right) \left( -%
\frac{K_{2,\frac{3}{2}}}{4\pi ^{2}c_{k}R^{3}r_{k}^{3}}\right)-\frac{K_{2,%
\frac{1}{2}}}{4\pi ^{2}c_{k}Rr_{k}}\frac{K_{2,1}}{4\pi
^{2}c_{k}R^{2}r_{k}^{2}}+O(\delta)\right) \\ &\cdot &\pi ^{2}R^{3}r_{k}^{3}=1-\frac{1}{c_{k}^{2}}\left( \frac{1}{12\pi ^{2}}\log \frac{-\lambda _{k}%
}{c_{k}^{2}}+\frac{1}{4\pi ^{2}}\log R+\frac{1}{8\pi ^{2}}+O(\delta)\right)\label{addb2}.
\end{eqnarray*}
Similarly, we can also obtain
\begin{eqnarray*}
III_{1} &=&\int_{\Omega\cap \partial B_{\delta }}\nu \left( \mathcal{W}_{k}\Delta ^{\frac{3}{2}%
}\mathcal{W}_{k}-\Delta ^{\frac{1}{2}}\mathcal{W}_{k}\Delta \mathcal{W}_{k}\right) d\sigma \\
&=&\frac{\nu}{c^2_{k}} \left(\left( \frac{-1}{4\pi ^{2}}\log \delta+O(\delta) \right)
\left( \frac{-K_{2,\frac{3}{2}}}{4\pi ^{2}\delta ^{3}}e_{3}(x-x_k)+O(\delta)\right) -%
\frac{-K_{2,\frac{1}{2}}e_{1}(x-x_k)}{4\pi ^{2}\delta }\frac{-K_{2,\frac{2}{2}}%
}{4\pi ^{2}\delta ^{2}}\right) \pi ^{2}\delta ^{3}+\frac{O(\delta)}{c^2_{k}} \\
&=&\frac{1}{c_{k}^{2}}\left( \frac{-1}{4\pi ^{2}}\log \delta -\frac{1}{8\pi
^{2}}+O(\delta)\right).
\end{eqnarray*}
Combining \eqref{addb1} and \eqref{addb2}, we derive that
\begin{eqnarray}\label{addb3}
\int_{A_{k}\left( R,\delta \right) }\left\vert \Delta \mathcal{W}_{k}\right\vert
^{2}dx &=&1-\frac{1}{c_{k}^{2}}\left( \frac{-1}{4\pi ^{2}}\log \delta -\frac{1}{%
8\pi ^{2}}+\frac{1}{12\pi ^{2}}\log \frac{-\lambda _{k}}{c_{k}^{2}}+\frac{1}{%
4\pi ^{2}}\log R+\frac{1}{8\pi ^{2}}+O(\delta)\right) \nonumber \\
&=&1-\frac{1}{c_{k}^{2}}\left( \frac{-1}{4\pi ^{2}}\log \frac{\delta }{R}+%
\frac{1}{12\pi ^{2}}\log \frac{-\lambda _{k}}{c_{k}^{2}}+O(\delta)\right).
\end{eqnarray}

Now, we are in position to give the sharp upper bound for $\lim\limits_{k\rightarrow\infty}\frac{-\lambda _{k}}{c_{k}^{2}}$. Indeed, from \eqref{addb4}, \eqref{addb3} and  \eqref{esti1}, we can get
\begin{equation*}\begin{split}
&\left\Vert \Delta u_{k}\right\Vert _{A_{k}\left( R,\delta \right)
}^{2}-\int_{A_{k}\left( R,\delta \right) }\left\vert \Delta \mathcal{W}_{k}\right\vert
^{2}dx \\
&\ \ =1-\frac{1}{c_{k}^{2}}\left( \frac{1}{16\pi ^{2}}+\frac{1}{4\pi ^{2}}\log
\frac{R}{2\delta }+\frac{1}{6\pi ^{2}}\log \left( \frac{\pi }{2}\right)
+A_{p}\right)  \\
&\ \ \ \ -1+\frac{1}{c_{k}^{2}}\left( \frac{-1}{4\pi ^{2}}\log \frac{\delta }{R}+%
\frac{1}{12\pi ^{2}}\log \frac{-\lambda _{k}}{c_{k}^{2}}+O(\delta)\right)  \\
&\ \ =\frac{1}{c_{k}^{2}}\left( \frac{-1}{4\pi ^{2}}\log \frac{\delta }{R}+%
\frac{1}{12\pi ^{2}}\log \frac{-\lambda _{k}}{c_{k}^{2}}-\frac{1}{16\pi ^{2}}%
-\frac{1}{4\pi ^{2}}\log \frac{R}{2\delta }-\frac{1}{6\pi ^{2}}\log \left(
\frac{\pi }{2}\right) -A_{p}+O(\delta)\right) \\
&\ \ =\frac{1}{c_{k}^{2}}\left( \frac{-1}{4\pi ^{2}}\log \frac{\delta }{R}+%
\frac{1}{12\pi ^{2}}\log \frac{-\lambda _{k}}{c_{k}^{2}}-\frac{1}{16\pi ^{2}}%
-\frac{1}{4\pi ^{2}}\log \frac{R}{2\delta }-\frac{1}{6\pi ^{2}}\log \left(
\frac{\pi }{2}\right) -A_{p}+O(\delta)\right)  \\
&\ \ \geq\frac{1}{c_{k}^{2}}\left( \frac{1}{6\pi ^{2}}\log \frac{-\lambda _{k}%
}{c_{k}^{2}}-2A_{p}-\frac{1}{3\pi ^{2}}\log
\frac{\pi }{2}+o_{\delta,k}(1)+o(R^{-1})\right)+2-\frac{\alpha_k}{6\pi^2},
\end{split}\end{equation*}
which implies
\begin{equation*}
\lim_{k\rightarrow\infty}\frac{-\lambda _{k}}{c_{k}^{2}}\leq 2\pi ^{2}\exp \left( -\frac{3}{4}+12\pi ^{2}A_{p}\right).
\end{equation*}

Therefore, we can conclude the following
\begin{proposition}
\label{main proposition}
 If $c_k  \to \infty$, then $$\mathop {\sup }\limits_{u \in W^{2,2} \left( \Omega \right),\left\| \Delta u \right\|_{2}  \leqslant 1} \int_{\partial\Omega}  {e^{12\pi^2 u^2 } dx}  \leqslant \left| \partial\Omega  \right| + 2\pi^2 e^{12\pi^2 A_p-\frac{3}{4}}.$$
\end{proposition}

\section{Existence of extremal functions}

 In this section, we assume $A_p=\max_{p\in\partial\Omega}A_p$ for some $p\in\partial\Omega$.
 Now we construct a blowing up sequence $\phi_{\varepsilon}$ with $\int_{\Omega}\left|\Delta \phi_{\varepsilon}\right|^{2}=1$, and
$$
\int_{\partial \Omega} e^{12\pi^2\left(\phi_{\varepsilon}-\overline{\phi_{\varepsilon}}\right)^{2}}d\sigma> |\partial \Omega|+2 \pi^2 e^{12\pi^2 A_{p}-\frac{3}{4}}, \quad \text { where } \overline{\phi_{\varepsilon}}=\frac{1}{| \Omega|} \int_{ \Omega} \phi_{\varepsilon}dx.
$$
Take a Fermi coordinate system $(U, \theta)$ around $p$ such that $\theta(p)=(0,0)$, $\theta$ maps $\partial \Omega \cap U$ inside $\partial \mathbb{R}_{+}^{4}$, and for any $\varepsilon>0$ and $x \in \partial\Omega$, there exists $\delta>0$ such that
$$
(1-\varepsilon) \theta \leq g=g_{i j} d x_{i} \otimes d x_{j}+d t \otimes d t \leq(1+\varepsilon) \theta \quad \text { in } M_{\delta},
$$where $M_{\delta}=\left\{x \in \Omega_{\delta}:\right.$ $\operatorname{dist}(\pi(x), p) \leq \delta\}$.

Set
$$
\tilde{\phi}_{\varepsilon}(x', t)=C+\frac{-(1 / 8 \pi^2) \log \left((\frac{\pi}{2})^{4/3} |x'|^{2} / \varepsilon^{2}+((\frac{\pi}{2})^{2/3} t / \varepsilon+1)^{2}\right)+B+g_\varepsilon(x',t)}{C}
$$
for some constants $B, C$, where $g_\varepsilon(x',t)=\frac{1}{2^{\frac{8}{3}}\pi^{\frac{4}{3}}}\frac{t/\varepsilon}{  \left(1+\left({\frac{\pi}{2}}\right)^{\frac{2}{3}}\frac{t}{\varepsilon}  \right)^2+  \left({\frac{\pi}{2}}\right)^{\frac{4}{3}}\frac{|x'|^2}{\varepsilon^2}  }$.

Let $B_{r}^{+}=B_{r}(p)\cap\Omega$ and $R$ be a function of $\varepsilon$ such that $R \rightarrow+\infty$ and $R \varepsilon \rightarrow 0$ as $\varepsilon \rightarrow 0$. Set
$$
\phi_{\varepsilon}= \begin{cases}\tilde{\phi}_{\varepsilon} \circ \theta(x) & \text { if } x \in B_{R \varepsilon}^{+}, \\ (G-\eta\beta) / C & \text { if } x \in B_{2R \varepsilon}^{+} \backslash B_{R \varepsilon}^{+},\\G/C & \text { if } x \in \Omega \backslash B_{2R \varepsilon}^{+},\end{cases}
$$
where $A, C$ are constants to be defined later, $\beta=G-C\tilde{\phi}_{\varepsilon} \circ \theta(x)$, $\eta$ is some radial function in  $C^{\infty}_0(B_{2R\varepsilon}(p))$ with $\eta\equiv1$ on $B_{R\varepsilon}(p)$, and $ |\nabla\eta|=O(\frac{1}{R\varepsilon}),|\Delta\eta|=O(\frac{1}{(R\varepsilon)^2})$. One can easily verify that $\frac{\partial \phi_{\varepsilon}(x)}{\partial \nu}=0$ for any $x\in\partial\Omega$.

Now, we estimate $\int_{\Omega}|\Delta\phi_\varepsilon|^2dx$, rewrite it as
\begin{align}\label{addc6}\int_{\Omega}|\Delta\phi_\varepsilon|^2dx=\left(\int_{B^{+}_{R\varepsilon}}+\int_{\Omega\setminus B^{+}_{R\varepsilon}}\right)|\Delta\phi_\varepsilon|^2dx:=I_1+I_2.\end{align}

Since
\begin{align}\label{addc7}
I_2&=\int_{\Omega\backslash B_{R\varepsilon}^{+}}\left\vert \Delta\phi
_{\varepsilon}\right\vert ^{2}  =\int_{\Omega\backslash B_{R\varepsilon
}^{+}}\frac{\left\vert \Delta G\right\vert ^{2}}{C^{2}}+\int_{B_{2R\varepsilon
}^{+}\backslash B_{R\varepsilon}^{+}}\frac{\left\vert \Delta\left(
\eta\left(  G-C\tilde{\phi}_{\varepsilon} \circ \theta(x)\right)  \right)  \right\vert ^{2}}{C^{2}%
}\nonumber \\&-\frac{2}{C^{2}}\int_{B_{2R\varepsilon}^{+}\backslash B_{R\varepsilon}^{+}%
}\left\vert \nabla G\nabla\left(  G-C\tilde{\phi}_{\varepsilon} \circ \theta(x)\right)  \right\vert
^{2}\nonumber\\& :=II_{1}+II_{2}+II_{3}.%
\end{align}

Let $C$ satisfy%

\begin{equation}\label{addc9}
C+\frac{-\frac{1}{8\pi^{2}}\log\left(  \left(  \frac{\pi}{2}\right)
^{4/3}R^{2}\right)  +B}{C}=\frac{-\frac{1}{4\pi^{2}}\log R\varepsilon+A_{p}%
}{C},%
\end{equation}
by direct computing, one can easily verify that \begin{equation}\label{addc4}\left\vert II_{2}\right\vert ,\left\vert
II_{3}\right\vert =\frac{1}{C^{2}}\left(O\left(  R\varepsilon\right) +O\left( R^{-1}\right) \right).\end{equation}
Similar as \eqref{add4} and \eqref{addc1}, we can obtain
\begin{align}\label{addc5}
II_{1}  &  =\int_{\Omega\backslash B_{R\varepsilon}^{+}}\frac{\left\vert \Delta
G\right\vert ^{2}}{C^{2}}=\int_{\partial\left(  \Omega\backslash
B_{R\varepsilon}^{+}\right)  }\nu\left(  -G\Delta^{3/2}G+\Delta^{1/2}G\Delta
G\right)  d\sigma  \nonumber \\
&  =\frac{1}{C^{2}}\left(  -\frac{1}{4\pi^{2}}\log R\varepsilon-\frac{1}%
{8\pi^{2}}+A_{p}+O\left(  R\varepsilon\right)  \right)
\end{align}
and
\begin{align}\label{addc8}
I_1=\int_{B^{+}_{R\varepsilon}}\left\vert \Delta\phi_{\varepsilon}\right\vert
^{2}=\frac{1}{C^{2}}\left(  \frac{3}{16\pi^{2}}+\frac{1}{4\pi^{2}}\log\frac
{R}{2}+\frac{1}{6\pi^{2}}\log\frac{\pi}{2}+O\left(  \frac{\log R}{R}\right).
\right)
\end{align}
Combining \eqref{addc6},\eqref{addc7},\eqref{addc4},\eqref{addc5} and \eqref{addc8}, we have%

\begin{align*}
\int_{\Omega}\left\vert \Delta\phi_{\varepsilon}\right\vert ^{2} &  =\frac
{1}{C^{2}}\left(  \frac{3}{16\pi^{2}}+\frac{1}{4\pi^{2}}\log\frac{R}{2}%
+\frac{1}{6\pi^{2}}\log\frac{\pi}{2}+O\left(  \frac{\log R}{R}\right)\right)  \\
&  +\frac{1}{C^{2}}\left(  -\frac{1}{4\pi^{2}}\log R\varepsilon-\frac{1}%
{8\pi^{2}}+A_{p}+O\left(  R\varepsilon\right)  +O\left(
\frac{1}{R}\right)  \right)  \\
&  =\frac{1}{C^{2}}\left(  \frac{1}{16\pi^{2}}+\frac{1}{6\pi^{2}}\log\frac
{\pi}{2}+\frac{1}{4\pi^{2}}\log\frac{1}{2\varepsilon}+A_{p}\right.  \\
&  \left.  +O\left(  R\varepsilon\right)  +O\left(  \frac
{\log R}{R}\right)  \right)  .
\end{align*}
To ensure that $\int_{\Omega}\left\vert \Delta\phi_{\varepsilon}\right\vert
^{2}=1$, we set%

\begin{equation}\label{addc2}
C^{2}=\frac{1}{16\pi^{2}}+\frac{1}{6\pi^{2}}\log\frac{\pi}{2}+\frac{1}%
{4\pi^{2}}\log\frac{1}{2\varepsilon}+A_{p}+O\left(  R\varepsilon\right)  +O\left(  \frac{\log R}{R}\right).
\end{equation}

On the other hand, from (\ref{addc9}), we have%

\[
C^{2}=-\frac{1}{4\pi^{2}}\log\varepsilon+A_{p}-B+\frac{1}{6\pi^{2}}\log
\frac{\pi}{2}.
\]
Therefore,
\begin{equation}\label{addc3}
B=\frac{1}{4\pi^{2}}\log2-\frac{1}{16\pi^{2}}+O\left(  R\varepsilon\right)  +O\left(  \frac{\log R}{R}\right).
\end{equation}

A straightforward computation gives%

\begin{align*}
\bar{\phi}_{\varepsilon} & =\frac{1}{\left\vert \Omega\right\vert }%
\int_{\Omega}\phi_{\varepsilon}=\frac{1}{C}\left(  O\left(  \left(
R\varepsilon\right)  ^{4}\log R\right)  +O\left(  \left(  R\varepsilon\right)
^{4}\log\varepsilon\right)  +O\left(  \left(  R\varepsilon\right)  ^{4}\log
R\varepsilon\right)  \right)  \\
& =\frac{1}{C}\left(  O\left(  \left(  R\varepsilon\right)  ^{4}\log R\right)
+O\left(  \left(  R\varepsilon\right)  ^{4}\log\varepsilon\right)  \right).
\end{align*}

Then
\begin{align*}
& \int_{\partial \Omega}\exp\left(
12\pi^{2}\left(  \phi_{\varepsilon}-\bar{\phi}_{\varepsilon}\right)
^{2}\right)  d\sigma\\
& \geq \int_{\partial B_{R\varepsilon}^{+}\cap\partial \mathbb{R}_{4}^{+}}\exp\left(
12\pi^{2}\left(  \tilde{\phi}_{\varepsilon}-\bar{\phi}_{\varepsilon}\right)
^{2}\left(  x',t\right)  \right)  dx'dt\\
&  \geq\int_{\partial B_{R\varepsilon}^{+}\cap\partial \mathbb{R}_{4}^{+}}\exp\left(
12\pi^{2}C^{2}-3\log\left(  \left(  \frac{\pi}{2}\right)  ^{4/3}%
\frac{\left\vert x'\right\vert ^{2}}{\varepsilon^{2}}+1\right)  +24\pi
^{2}B-24\pi^{2}C\bar{\phi}_{\varepsilon}\right)  dx'\\
&  =\exp\left(  12\pi^{2}C^{2}+24\pi^{2}B+O\left(  \left(  R\varepsilon
\right)  ^{4}\log R\right)  +O\left(  \left(  R\varepsilon\right)  ^{4}%
\log\varepsilon\right)  \right)  \int_{B_{R\varepsilon}^{3}}\frac{1}{\left(
\left(  \frac{\pi}{2}\right)  ^{4/3}\frac{\left\vert x'\right\vert ^{2}%
}{\varepsilon^{2}}+1\right)  ^{3}}dx'.
\end{align*}
Let $\left(  \frac{\pi}{2}\right)  ^{2/3}\frac{x'}{\varepsilon}=\tilde{x}$.
Then
\begin{align*}
\int_{B_{R\varepsilon}^{3}}\frac{1}{\left(  \left(  \frac{\pi}{2}\right)
^{4/3}\frac{\left\vert x'\right\vert ^{2}}{\varepsilon^{2}}+1\right)  ^{3}}dx'
&  =\left(  \frac{2}{\pi}\right)  ^{2}\varepsilon^{3}\int_{B_{\left(
\frac{\pi}{2}\right)  ^{2/3}R}^{3}}\frac{1}{\left(  \tilde{x}^{2}+1\right)
^{3}}d\tilde{x}\\
&  =\left(  \frac{2}{\pi}\right)  ^{2}\varepsilon^{3}\int_{0}^{\left(
\frac{\pi}{2}\right)  ^{2/3}R}\frac{4\pi r^{2}}{\left(  r^{2}+1\right)  ^{3}%
}dr\\
&  =\left(  \frac{2}{\pi}\right)  ^{2}\varepsilon^{3}4\pi\int_{0}^{\left(
\frac{\pi}{2}\right)  ^{2/3}R}\frac{r^{2}}{\left(  r^{2}+1\right)  ^{3}}dr\\
&  =\varepsilon^{3}\left(  1+O\left(  \frac{1}{R}\right)  \right)  ,
\end{align*}
where we have used the fact $\int_{0}^{\infty}\frac{r^{2}}{\left(
r^{2}+1\right)  ^{3}}dr=\frac{1}{16}\pi$.

Hence, it follows from \eqref{addc2} and \eqref{addc3} that %
\begin{align*}
&  \int_{\partial B_{R\varepsilon}^{+}\cap\partial \mathbb{R}_{4}^{+}}\exp\left(
12\pi^{2}\left(  \phi_{\varepsilon}-\bar{\phi}_{\varepsilon}\right)
^{2}\left(  x',t\right)  \right)  dx'dt\\
&  \geq\varepsilon^{3}\left(  1+O\left(  \frac{1}{R}\right)  \right)  \exp\left(
12\pi^{2}C^{2}+24\pi^{2}B+O\left(  \left(  R\varepsilon\right)  ^{4}\log
R\right)  +O\left(  \left(  R\varepsilon\right)  ^{4}\log\varepsilon\right)
\right)  \\
&  =\varepsilon^{3}\exp\left(  12\pi^{2}\left(  \frac{1}{16\pi^{2}}+\frac
{1}{6\pi^{2}}\log\frac{\pi}{2}+\frac{1}{4\pi^{2}}\log\frac{1}{2\varepsilon
}+A_{p}\right)  \right.  \\
&  \left.  +24\pi^{2}\left(  \frac{1}{4\pi^{2}}\log2-\frac{1}{16\pi^{2}%
}\right)  \right)  + O(R\varepsilon) +O\left(  \frac{\log R}{R} +O\left(  \left(
R\varepsilon\right)  ^{4}\log R\right)  +O\left(  \left(  R\varepsilon\right)
^{4}\log\varepsilon\right)  \right)  \\
&  =\exp\left(  -\frac{3}{4}+2\log\pi+\log2+12\pi^{2}A_{p}\right)  +O\left(
\left(  R\varepsilon\right)  ^{4}\log R\right)  +O\left(  \left(
R\varepsilon\right)  ^{4}\log\varepsilon\right)  +O(R\varepsilon)+O\left(  \frac{\log R}{R}\right)
\\
&  =2\pi^{2}\exp\left(  -\frac{3}{4}+12\pi^{2}A_{p}\right)  +O\left(  \left(
R\varepsilon\right)  ^{4}\log R\right)  +O\left(  \left(  R\varepsilon\right)
^{4}\log\varepsilon\right) +O(R\varepsilon) +O\left(  \frac{\log R}{R}\right)  .
\end{align*}

Moreover, we have
\begin{align*}
\int_{\Omega\backslash\partial B_{R\varepsilon}^{+}}\exp\left(  12\pi
^{2}\left(  \phi_{\varepsilon}-\bar{\phi}_{\varepsilon}\right)  ^{2}\right)
d\sigma  & \geq\int_{\partial\Omega\backslash\partial B_{R\varepsilon}^{+}%
}\left(  1+12\pi^{2}\left(  \phi_{\varepsilon}-\bar{\phi}_{\varepsilon
}\right)  ^{2}\right)  d\sigma\\
& \geq\left\vert \partial\Omega
\backslash\partial B_{R\varepsilon}^{+}\right\vert +\frac{12\pi^{2}}{C^{2}%
}\int_{\partial\Omega\backslash\partial B_{2R\varepsilon}^{+}}\left(
G-C\bar{\phi}_{\varepsilon}\right)  ^{2}d\sigma.%
\end{align*}
Therefore,
\begin{align*}
& \int_{\partial\Omega}\exp\left(  12\pi^{2}\left(  \phi_{\varepsilon}%
-\bar{\phi}_{\varepsilon}\right)  ^{2}\right)  d\sigma\\
& \geq\left\vert \partial\Omega\right\vert  -O\left(  \left(  R\varepsilon\right)  ^{3}\right) +\frac{12\pi^{2}}{C^{2}}%
\int_{\partial\Omega\backslash\partial B_{2R\varepsilon}^{+}}\left(
G-C\left(  O\left(  \left(  R\varepsilon\right)  ^{4}\log R\right)  +O\left(
\left(  R\varepsilon\right)  ^{4}\log\varepsilon\right)  \right)  \right)
^{2}d\sigma  \\
& +2\pi^{2}\exp\left(  -\frac{3}{4}+12\pi^{2}A_{p}\right)  +O\left(  \left(
R\varepsilon\right)  ^{4}\log R\right)  +O\left(  \left(  R\varepsilon\right)
^{4}\log\varepsilon\right)  +O\left(  \frac{\log R}{R}\right)+O(R\varepsilon)  \\
& =\left\vert \partial\Omega\right\vert +2\pi^{2}\exp\left(  -\frac{3}%
{4}+12\pi^{2}A_{p}\right)  +\frac{12\pi^{2}}{C^{2}}\int_{\partial\Omega}%
G^{2}+O\left(  \left(  R\varepsilon\right)  ^{4}\log R\right)  \\
& +O\left(  \left(  R\varepsilon\right)  ^{4}\log\varepsilon\right) +O(R\varepsilon) +O\left(
\frac{\log R}{R}\right).
\end{align*}

Let $R=\log^2\varepsilon$. Then we have $R\longrightarrow\infty$ and $R\varepsilon
\longrightarrow0$, and
\[
\left(  R\varepsilon\right)  ^{4}\log R+\left(  R\varepsilon\right)  ^{4}%
\log\frac{1}{\varepsilon}+O\left(  \frac{\log R}{R}\right) +O(R\varepsilon) =o\left(  \frac
{1}{C^{2}}\right)  .
\]
Hence%

\[
\int_{\partial\Omega}\exp\left(  12\pi^{2}\left(  \phi_{\varepsilon}-\bar
{\phi}_{\varepsilon}\right)  ^{2}\right)  d\sigma>\left\vert \partial
\Omega\right\vert +2\pi^{2}\exp\left(  -\frac{3}{4}+12\pi^{2}A_{p}\right),
\]
as $\varepsilon$ is small enough.

\end{document}